\documentclass[12pt]{amsart}
\usepackage{amsmath,amsfonts,amssymb,latexsym}
\usepackage{hyperref}
\usepackage{graphicx}
\usepackage[a4paper, left=2cm, right=2cm, top=2cm, bottom=2cm]{geometry}

\newcommand{\NN}{\mathbb{N}}
\newcommand{\CC}{\mathbb{C}}
\newcommand{\RR}{\mathbb{R}}

\newcommand{\EE}{\mathbb{E}}
\newcommand{\Oo}{\mathcal{O}}
\newcommand{\Ss}{\mathcal{S}}
\newcommand{\Bo}{\mathcal{B}}
\newcommand{\La}{\mathcal{L}}

\newcommand{\dist}{{\rm dist\,}}

\newtheorem{Thm}{Theorem}
\newtheorem{Lem}{Lemma}
\newtheorem{Prop}{Proposition}

\theoremstyle{definition}
\newtheorem{Rem}{Remark}
\newtheorem{Def}{Definition}
\newtheorem{Ex}{Example}

\begin{document}
\title[Sequences of generalized moments]{Sequences of generalized moments and their applications to some moment partial differential equations}
\author{S{\l}awomir Michalik}
  \address{{Faculty of Mathematics and Natural Sciences,
College of Science,
Cardinal Stefan Wyszy{\'n}ski University,
W\'oycickiego 1/3,
01-938 Warszawa, Poland}\newline
{\& Shibaura Institute of Technology,
Department of Engineering and Design,
Saitama 337-8570, Japan}\newline
ORCiD: 0000-0003-4045-9548}
\email{s.michalik@uksw.edu.pl}
\author{Hiroshi Yamazawa}
\address{Shibaura Institute of Technology,
Department of Engineering and Design,
Saitama 337-8570, Japan}
\email{yamazawa@shibaura-it.ac.jp}
\begin{abstract}
We introduce sequences of generalized moments that are a generalization of sequences preserving summability and sequences of moments from Balser's theory of moment summability. We construct extensions of Laplace and Borel integral operators for these new sequences. We find integral representations of moment derivatives for generalized moments.
We show the characterization of summable solutions of moment-PDEs with generalized moments.
We describe singular directions for summable solutions of moment-PDEs with generalized moments and we express jumps across these singular directions in terms of hyperfunctions and in terms of kernel functions associated with sequences preserving summability, which are connected with given sequences of generalized moments.  
\end{abstract}

\maketitle
\section{Introduction}\label{sec:1}
In this paper we would like to propose a new and unified approach to sequences of moments that have arisen from Balser's theory of moment summability \cite[Section 6.5]{B2} and to sequences preserving summability introduced recently in \cite{I-M} and developed in \cite{L-M} and \cite{M-S-T}. Namely, we introduce sequences of generalized moments that are products of sequences of moments and sequences preserving summability. 

For such sequences, in particular, we may find an appropriate extension of Laplace and Borel integral operators and we can construct integral representations of moment derivative of holomorphic functions in the complex neighborhood of the origin.

Sequences of generalized moments are particularly fruitful tool 
for studying summable power series solutions of some linear $q$-difference-differential equations and their extensions to $q$-difference-fractional differential equations
in a general framework of moment partial differential equations for generalized moments. 

The important result of this type is given in Theorem \ref{th:4}, which reduces the problem of characterization of summable solutions for some moment partial differential equations with generalized moments to the same problem for summable solutions of appropriate moment partial differential equations with moments in Balser's sense. 

This new approach is especially useful for the study of the Stokes phenomenon for summable solutions to some moment partial differential equations for generalized moments in the similar spirit to \cite{M-P} and \cite{M-T}. The main results in these directions are given in Theorems \ref{th:5} and \ref{th:6}, where jumps across the singular directions for summable solutions of some moment partial differential equations for generalized moments are calculated. Next, we also calculate explicitly jumps for solutions of some simple $q$-difference-differential and $q$-difference-fractional differential equations (see Examples \ref{ex:4}--\ref{ex:6}). 
It is worth emphasizing that this description of the Stokes phenomenon was our primary motivation for studying sequences of generalized moments. 

This paper also represents a new advance in the theory of sequences preserving summability. Namely, we introduce here the kernel functions $e_m(z)$ and $E_m(z)$ associated with a given sequence $m$ preserving summability, which seem to be new and promising tools in the theory of summability for generalized moments. In particular we use them to construct Borel and Laplace integral operators for sequences preserving summability and for sequences of generalized moments. These kernels are also useful in the description of the Stokes phenomenon for summable solutions of some moment partial differential equations connected with sequences of generalized moments.

The paper is organized as follows. In Sections~\ref{sec:2}--\ref{sec:5} we recall necessary notation, definitions and previously known results connected with theory of summability and with sequences preserving summability. In Section~\ref{sec:6} we define kernel functions associated with a sequence $m$ preserving summability and we describe their properties and applications. In the next section we introduce sequences of generalized moments. Next, in Section \ref{sec:8} we extend Laplace and Borel integral operators to sequences of generalized moments. In the subsequent section we recall the notion of moment differentiation and we find an integral representation of $m$-moment derivative for $m$ being a sequence of generalized moments (Theorem \ref{th:3}). Next, in Section \ref{sec:10} we characterize summable solutions to some moment partial differential equations connected with sequences of generalized moments (Theorem \ref{th:4}). In the last two sections we recall the concept of the Stokes phenomenon for summable formal power series
and we describe jumps across singular directions for summable solutions of some moment partial differential equations connected with sequences of generalized moments (Theorems \ref{th:5} and \ref{th:6}).

\section{Notations}\label{sec:2}
Let $\EE$ stand for a complex Banach space with a norm $\|\cdot\|_{\EE}$. The vector space of formal power series with coefficients in $\EE$ is denoted by $\EE[[t]]$.
For any given set $G\subseteq\CC^n$, by $\Oo(G,\EE)$ we denote the set of all $\EE$-valued holomorphic functions defined on some open set $U$ containing $G$. We also write $\Oo(G)$ if $\EE=\CC$ for simplicity.

An open disc in $\CC$ of radius $r>0$ with a center at the origin is denoted by $D_r$. 
 In case when the radius $r$ is not essential, the set $D_r$ is briefly denoted by $D$. By $\Oo_r$ we denote the Banach space of holomorphic functions on the disc $D_r$, continuous on its closure $\overline{D}_r$ and equipped with a norm $\|f\|_{\Oo_r}:=\sup_{z\in D_r}|f(z)|$. If the radius $r$ is not essential the Banach space $\Oo_r$ is denoted briefly by $\Oo$. Observe that by the Hartogs theorem on separate analyticity we may rewrite the space $\Oo(G,\EE)$ for $\EE=\Oo$ as $\Oo(G\times D)$.

An~\emph{unbounded sector $S$ in a direction $d\in\RR$ with an opening $\alpha>0$} in the universal covering space $\widetilde{\CC}$ of
$\CC\setminus\{0\}$ is defined by
\begin{displaymath}
S=S_d(\alpha):=\{z\in\widetilde{\CC}\colon\ z=r\*e^{i\phi},\ r>0,\ \phi\in(d-\alpha/2,d+\alpha/2)\}.
\end{displaymath}
If the opening $\alpha$ is not essential,
 the sector $S_d(\alpha)$ is denoted briefly by $S_d$.
 
 A \emph{bounded sector} $S(d, \alpha, R)$ is defined by  
 $$S(d, \alpha, R):= \left\{z\in\widetilde{\mathbb{C}}\colon z=re^{i\phi},\ r \in (0, R),\ \phi \in \left(d-\frac{\alpha}{2}, d+\frac{\alpha}{2}\right)\right\}.$$
 
 We say that $G=G(d,\alpha)\subseteq \widetilde{\CC}$ is a \emph{sectorial region of opening $\alpha>0$ and bisecting direction $d\in\RR$} if $G\subseteq S_d(\alpha)$ and for every $\beta\in(0,\alpha)$ there exists $r>0$ such that $S(d,\beta,r)\subseteq G$.  
 
  We also briefly denote a \emph{disc-sector} $S_d(\alpha)\cup D_r$ by $\hat{S}_d(\alpha,r)$ (resp. $S_d(\alpha)\cup D$ by $\hat{S}_d(\alpha)$ and   $S_d\cup D$ by $\hat{S}_d$).

For given $s\geq0$ we will denote by $\Gamma_s$ the sequence $(\Gamma(1+sn))_{n\geq0}$, where $\Gamma(\cdot)$ is the Euler gamma function.

Throughout the paper we assume that $m=(m(n))_{n\geq 0}$ (resp. $\tilde{m}=(\tilde{m}(n))_{n\geq 0}$) is a sequence of real positive numbers with $m(0)=1$ (resp. $\tilde{m}(0)=1$).

\section{$q$-calculus} \label{sec:3} 
In this section we introduce the basic notion of $q$-calculus following \cite{G-R}. In the whole paper we assume that  $q\in[0,1)$. 
If $u\in\EE[[t]]$ or $u\in\Oo(D,\EE)$ then we define the \emph{$q$-difference operator} $D_{q,t}$ as
\begin{equation*}
D_{q,t}u(t):=\frac{u(qt)-u(t)}{qt-t}.
\end{equation*}
For every $n\in\NN_{0}$ we define a $q$-analog of $n$ by 
\begin{equation*}
 [n]_q:=1+q+\dots+q^{n-1}=\frac{1-q^n}{1-q}.
\end{equation*}
We also introduce a $q$-analog of the factorial $n!$
\begin{equation*}
 [n]_q!:=
    \left\{
    \begin{array}{ll}
     1&\textrm{for}\ n=0\\  
     {[1]_q\cdots [n]_q}&\text{for}\ n\geq 1
     \end{array}
    \right..
\end{equation*}
 For every $n\in\NN_0\cup\{\infty\}$ and $a\in\CC$ we define $q$-shift factorial by
\begin{equation*}
 (a;q)_n:=
    \left\{
    \begin{array}{ll}
     1&\text{for}\ n=0\\
     \prod_{j=0}^{n-1}(1-aq^j)&\text{for}\ n\in\NN\\
     \prod_{j=0}^{\infty}(1-aq^j)&\text{for}\ n=\infty
     \end{array}
    \right..
\end{equation*}
Observe that the infinity product $(a;q)_{\infty}$ is convergent for any $a\in\CC$. We also set 
\begin{equation*}
 (a_1,\dots,a_r;q)_n:=\prod_{j=1}^r(a_j;q)_n\quad\text{for}\quad a_1,\dots,a_r\in\CC\quad\textrm{and}\quad n\in\NN_0\cup\{\infty\}.
\end{equation*}
The basic hypergeometric series is defined as
\begin{equation*}
 {}_{k+1}\phi_k\left(\begin{array}{c}
                       a_1,\dots,a_{k+1}\\
                       b_1,\dots,b_k
                     \end{array}
;q,x\right):=\sum_{n=0}^{\infty}\frac{(a_1,\dots,a_{k+1};q)_n}{(b_1,\dots,b_k;q)_n(q;q)_n}x^n.
\end{equation*}

In the paper we will use the following two fundamental formulas for the basic hypergeometric series
\begin{Prop}[$q$-binomial theorem, {\cite[Section 1.3]{G-R}}]
\label{prop:1}
If $|q|<1$ and $|z|<1$ then
\begin{equation*} 
 {}_{1}\phi_0\left(\begin{array}{c}
                       a\\
                       -
                     \end{array}
;q,z\right)=\sum_{n=0}^{\infty}\frac{(a;q)_n}{(q;q)_n}z^n=\frac{(az;q)_{\infty}}{(z;q)_{\infty}}.
\end{equation*}
\end{Prop}

\begin{Prop}[Heine's transformation formula, {\cite[Section 1.4]{G-R}}]
\label{prop:2}
If $|b|<1$ and $|z|<1$ then 
\begin{equation*} 
 {}_{2}\phi_1\left(\begin{array}{c}
                       a,b\\
                       c
                     \end{array}
;q,z\right)=\sum_{n=0}^{\infty}\frac{(a,b;q)_n}{(c;q)_n(q;q)_n}z^n=\frac{(b,az;q)_{\infty}}{(c,z;q)_{\infty}}{}_{2}\phi_1\left(\begin{array}{c}
                       c/b,z\\
                       az
                     \end{array}
;q,b\right).
\end{equation*}
\end{Prop}

\section{Summability and multisummability}\label{sec:4}
We introduce $k$-summable and multisummable power series. For more details about summability we refer the reader to \cite{B1,B2,LR}.

\begin{Def}[see {\cite[Section 5.2]{B2}}]
\label{df:borel}
 For a fixed sequence $m=(m(n))_{n\geq 0}$ of positive numbers, a linear operator $\hat{\Bo}_{m}\colon \EE[[t]]\to\EE[[t]]$ defined by
 \begin{equation*}
  (\hat{\Bo}_{m}\hat{u})(t):=
  \sum_{n=0}^{\infty}\frac{a_n}{m(n)}t^n\quad\text{for}\quad\hat{u}(t)=\sum_{n=0}^{\infty}a_nt^n\in\EE[[t]]
 \end{equation*}
 is called an \emph{$m$-Borel operator}.
\end{Def}
\begin{Rem}
Observe that for a given sequence $m=(m(n))_{n\geq 0}$ of positive numbers with $m(0)=1$, an inverse $m$-Borel operator $\hat{\Bo}_{m}^{-1}\colon\EE[[t]]\to\EE[[t]]$, called sometimes an \emph{$m$-Laplace operator}, is given by
 $\hat{\Bo}_{m}^{-1}= \hat{\Bo}_{m^{-1}}$ on $\EE[[t]]$, where $m^{-1}=(m(n)^{-1})_{n\geq 0}$. 
 Hence an $m$-Borel operator $\hat{\Bo}_{m}$ is a linear automorphism on the space of formal power series $\EE[[t]]$.
 \end{Rem}
\begin{Def}
\label{df:growth}
A function $f\in\Oo(\hat{S}_d(\alpha,r),\EE)$
is of \emph{exponential growth of order at most $k>0$ as $t\to\infty$
in $\hat{S}_d(\alpha,r)$} if for any
$\tilde{\alpha}\in(0,\alpha)$ and $\tilde{r}\in(0,r)$ there exist positive constants
$A,B$ such that
\begin{equation*}
\|f(t)\|_{\EE}<Ae^{B|t|^k}\quad \textrm{for every}\ t\in \hat{S}_d(\tilde{\alpha},\tilde{r}).
\end{equation*}
The set $\hat{S}_d(\tilde{\alpha},\tilde{r})$ is called a \emph{sub-disc-sector of $\hat{S}_d(\alpha,r)$}.

The space of such functions is denoted by $\Oo^k(\hat{S}_d(\alpha,r),\EE)$.
\end{Def}

\begin{Def}
We say that a function $f\in\Oo(\hat{S}_d,\EE)$
is of \emph{less than exponential growth as $t\to\infty$
in $\hat{S}_d$} if $f\in\Oo^k(\hat{S}_d,\EE)$ for every $k>0$.

We will denote the space of such functions by $\Oo^{>0}(\hat{S}_d,\EE)$.
\end{Def}

We recall the definitions of Laplace and Borel integral operators $\La_{\Gamma_{1/k},d}$ and $\Bo_{\Gamma_{1/k},d}$.
\begin{Def}[{see \cite[Chapter 2]{B1} and \cite[Chapter 5]{B2}}]\label{def:LB_operators}
  Fix $k>0$ and $d\in\RR$.
  \begin{itemize}
  \item If $v\in\Oo^{k}(\widehat{S}_d,\EE)$ then the integral operator $\La_{\Gamma_{1/k},d}$ defined by
  \begin{equation}\label{eq:lap}
   (\La_{\Gamma_{1/k},d}v)(t):=t^{-k}\int_{e^{id}\RR_+}v(s)e^{-(s/t)^k}\,d(s^k)=\int_{e^{id}\RR_+}v(s)e^{-(s/t)^k}k(s/t)^k\,\frac{ds}{s}
  \end{equation}
 is called the \emph{Laplace transform of order $k$ in a direction $d$}.
 \item If $u\in\Oo(S_d(\frac{\pi}{k}+\varepsilon,R),\EE)$ for some $\varepsilon>0$ and $R<\infty$ then the integral operator $\Bo_{\Gamma_{1/k},d}$ defined by
  \[
   (\Bo_{\Gamma_{1/k},d}u)(s):=\frac{1}{2\pi i}\int_{\gamma_k(d)}t^ku(t)e^{(s/t)^k}\,d(t^{-k})=-\frac{k}{2\pi i}\int_{\gamma_k(d)}u(t)e^{(s/t)^k}\,\frac{dt}{t}
  \]
(where a path $\gamma_k(d)$ is the boundary of a sector contained in $S_d(\frac{\pi}{k}+\varepsilon,R)$ with bisecting direction $d$,
a finite radius, an opening slightly larger than $\pi/k$, and the orientation is negative) is called the \emph{Borel transform of order $k$ in a direction $d$}.
 \end{itemize}
\end{Def}

\begin{Def}
\label{df:summable}
Let $k>0$ and $d\in\RR$. Then $\hat{u}\in\EE[[t]]$ is called
\emph{$k$-summable in a direction $d$} if there exists $\varepsilon>0$ and a disc-sector $\hat{S}_d=\hat{S}_d(\varepsilon)$ in a direction $d$ such that
$v:=\hat{\Bo}_{\Gamma_{1/k}}\hat{u}\in\Oo^k(\hat{S}_d,\EE)$.

Moreover, the \emph{$k$-sum of $\hat{u}$ in the direction $d$} is given by
\begin{equation}
\label{eq:sum}
       u^d(t)=\mathcal{S}_{k,d}\hat{u}(t):=
       (\La_{\Gamma_{1/k},\theta}v)(t)=t^{-k}\int_{e^{i\theta}\RR_+}e^{-(s/t)^k}v(s)\,d(s^k)
 \quad\textrm{for}\quad
       \theta\in(d-\varepsilon/2,d+\varepsilon/2).
\end{equation}

The space of $k$-summable formal power series in a direction $d$ is denoted by $\EE\{t\}_{k,d}$.
\end{Def}

\begin{Def}\label{df:summable_general}
 Let $k>0$. If $\hat{u}\in\EE[[t]]$ is $k$-summable in all directions $d$ but (after identification modulo $2\pi$) finitely many directions $d_1,\dots,d_n$ then $\hat{u}$ is called \emph{$k$-summable} and the directions $d_1,\dots,d_n$ are called the \emph{singular directions of $\hat{u}$}, provided they occur.
 
The space of $k$-summable formal power series is denoted by $\EE\{t\}_{k}$.
\end{Def}

\begin{Def}
\label{df:multisummable}
 Let $k_1>\dots>k_n>0$.
 We say that a real vector $\mathbf{d}=(d_1,\dots,d_n)\in\RR^n$ is an \emph{admissible multidirection with respect to} $\mathbf{k}=(k_1,\dots,k_n)$ if
 \begin{equation*}
 |d_{j}-d_{j-1}| \leq \frac{\pi}{2} \left(\frac{1}{k_{j}}-\frac{1}{k_{j-1}}\right)\qquad \text{for} \qquad j=2,\dots,n.
 \end{equation*}
\end{Def}

\begin{Def}
Let $\mathbf{k}$ and $\mathbf{d}$ be as in the previous definition. A series $\hat{u}\in\EE[[t]]$ is called \emph{$\mathbf{k}$-multisummable} in an admissible multidirection $\mathbf{d}$
if $\hat{u}=\hat{u}_1+\cdots+\hat{u}_n$ and $\hat{u}_j$ is $k_j$-summable in a direction $d_j$ for $j=1,\dots,n$. 
Then the \emph{$\mathbf{k}$-multisum of $\hat{u}$ in the multidirection $\mathbf{d}$} is given by 
$$u^{\mathbf{d}}(t)=\mathcal{S}_{\mathbf{k},\mathbf{d}}\hat{u}(t)=
\Ss_{k_1,d_1}\hat{u}_1(t)+\cdots+\Ss_{k_n,d_n}\hat{u}_n(t).
$$

Moreover, if additionally $\hat{u}_j$ is $k_j$-summable with $n_j$ singular directions $d_{j,1},\dots,d_{j,n_j}$ for $j=1,\dots,n$
then $\hat{u}$ is called \emph{$\mathbf{k}$-multisummable}
and $d_{j,1},\dots,d_{j,n_j}$ are called \emph{singular directions of $\hat{u}$ of level $k_j$}.

We denote the space of $\mathbf{k}$-multisummable series in an admissible multidirection $\mathbf{d}$ and the space of $\mathbf{k}$-multisummable series by $\EE\{t\}_{\mathbf{k},\mathbf{d}}$ and
$\EE\{t\}_{\mathbf{k}}$ respectively.
\end{Def}

\section{Sequences preserving summability}\label{sec:5}
We recall the concept of sequences preserving summability, introduced recently in \cite{I-M} and developed in \cite{M-S-T}. 

\begin{Def}[{\cite[Definition 11]{I-M}}]
\label{df:preserve}
  We say that a sequence $m=(m(n))_{n\geq 0}$ \emph{preserves summability} if for every $k>0$, $d\in\RR$ and every $\hat{u}\in\EE[[t]]$ the following equivalence holds:
  \begin{equation}\label{eq:equiv}
   \hat{u}\in\EE\{t\}_{k,d}\quad\text{if and only if}\quad
   \hat{\Bo}_{m}\hat{u}\in\EE\{t\}_{k,d}.
  \end{equation}
\end{Def}

\begin{Ex}\label{ex:1}
The following sequences $m=(m(n))_{n\geq0}$ preserve summability.
\begin{itemize}
 \item (See {\cite[Example 2]{I-M}}).\\
 If $a > 0$ and $\mathbf{a} := (a^n)_{n\geq 0}$ then the sequence $\mathbf{a}$ preserves summability. Indeed, in this case
 if $\hat{u}(t)=\sum_{n=0}^{\infty}u_nt^n$ then 
 $$\hat{\Bo}_{\mathbf{a}}\hat{u}(t)=\sum_{n=0}^{\infty}\frac{u_n}{a^n}t^n=\hat{u}\Big(\frac{t}{a}\Big)
 $$
 and since $a>0$ we conclude that the equivalence (\ref{eq:equiv}) is satisfied.
 In particular
the sequence $\mathbf{1} = (1)_{n\geq 0}$ preserves summability in a trivial way.
 \item (See {\cite[Example 2]{I-M}}).\\  Let $m(n)=\frac{\mathfrak{m}_1(n)}{\mathfrak{m}_2(n)}$ for $n\in\NN_0$, where $\mathfrak{m}_1(u)$ and $\mathfrak{m}_2(u)$ are sequences of moments in Balser's sense (see \cite[page 86]{B2}) of the same positive order.
 By Balser's theory of moment summability \cite[Section~6.5 and Theorem~38]{B2}, the sequence $(m(n))_{n\geq 0}$ preserves summability.
 
In particular w can take as such sequence
\begin{equation*}
m(n):=\frac{\Gamma(1+a_1 n)\cdots\Gamma(1+a_k n)}{\Gamma(1+b_1 n)\cdots \Gamma(1+b_l n)}\quad \text{for}\quad n\in\NN_0,
\end{equation*}
where $a_1,\dots,a_k$ and $b_1,\dots b_l$ are positive numbers satisfying
\begin{equation*}
 a_1+\cdots+a_k=b_1+\cdots+b_l,
\end{equation*}
and $\Gamma(\cdot)$ denotes the gamma function.
\item (See {\cite[Theorem 2]{I-M}}).\\
If $q\in[0,1)$ then the sequence $m=([n]_q!)_{n\geq 0}$ preserves summability.
\item (See {\cite[Theorem 10]{M-S-T}}).\\ Let
\begin{equation*}
 m(n)=w_1(n)a_1^n+\dots+w_k(n)a_k^n\quad\text{for}\quad n\in\NN_0,
\end{equation*}
where $a_1,\dots,a_k>0$ and $w_1,\dots,w_k$ are polynomials satisfying $w_1(n),\dots,w_k(n)\geq 0$ for $n\in\NN_0$ and $w_1(0)+\dots+w_k(0)=1$.
Then the sequence $m=(m(n))_{n\geq 0}$ preserves summability.
\end{itemize}
\end{Ex}

In the next theorem we collect known properties of sequences preserving summability.
\begin{Thm}\label{th:1}
 Sequences preserving summability have the following properties:
 \begin{enumerate}
  \item[(a)] (See \cite[Remark 9]{I-M}). If $m=m(n)_{n\geq0}$ preserves summability then there exists $a,A>0$ such that
  \begin{equation}
  \label{eq:aA}
   a^n\leq m(n) \leq A^n\quad\text{for every}\quad n\in\NN_0.
  \end{equation}
 \item[(b)] (See \cite[Remark 4]{M-S-T}). The set of sequences preserving summability forms a group with multiplication. If $m_1=(m_1(n))_{n\geq0}$ and $m_2=(m_2(n))_{n\geq 0}$ preserve summability then also their product $m=m_1\cdot m_2=(m_1(n) m_2(n))_{n\geq 0}$ preserves summability. Observe also that the identity element $\mathbf{1}=(1)_{n\geq 0}$ and the inverse element $m^{-1}=(m(n)^{-1})_{n\geq 0}$ to $m=(m(n))_{n\geq 0}$ preserve summability.
 \item[(c)] (Open set property, see \cite[Theorem 3]{M-S-T}.)
 Assume that $m=(m(n))_{n\geq0}$ is a sequence preserving summability and $\tilde{m}=(\tilde{m}(n))_{n\geq0}$ is a sequence of positive numbers with $m(0)=1$, which is close to $m$ in the sense that for every $k>0$ there exist $A_k,B_k>0$ satisfying
  \begin{equation}
  \label{eq:eq_2}
   \Big|m(n)-\tilde{m}(n)\Big|\leq\frac{A_kB_k^n}{\Gamma(1+\frac{n}{k})}\quad\text{for every}\quad n\in\NN_0.
  \end{equation}
 Then $\tilde{m}$ is also a sequence preserving summability.
 \end{enumerate}
\end{Thm}

The crucial role in the study of sequences preserving summability plays its following characterization
\begin{Thm}[{\cite[Theorem 1]{I-M}}]
 \label{th:2}
  A sequence $m=(m(n))_{n\geq 0}$ preserves summability if and only if
  \begin{equation}
  \label{eq:eq_1.5}
  \hat{\Bo}_{m}\big(\sum_{n=0}^{\infty}t^n\big)\in\Oo^{>0}(\CC\setminus\RR_+)\quad \textrm{and}\quad \hat{\Bo}_{m^{-1}}\big(\sum_{n=0}^{\infty}t^n\big)\in\Oo^{>0}(\CC\setminus\RR_+).
  \end{equation}
\end{Thm}

\begin{Rem}
 The idea of sequences preserving summability is similar in spirit to the concept of summability factor for $k$-summability introduced by Balser \cite[Exercise on page 109]{B2}. Namely, a sequence $\lambda=(\lambda(n))_{n\geq 0}$ is called a \emph{summability factor} for $k$-summability, if for  every $\hat{u}\in\EE\{t\}_k$, we have $\hat{\Bo}_{\lambda^{-1}}\hat{u}\in\EE\{t\}_k$. Observe that if $m$ is a sequence preserving summability then $m$ and $m^{-1}$ are summability factors for $k$-summability for every $k>0$.
\end{Rem}

\section{Kernel functions associated with sequences preserving summability}\label{sec:6}
Theorem \ref{th:2} suggests to consider the following kernel functions.
\begin{Def}\label{df:kernels}
Let $m=(m(n))_{n\geq0}$ be a sequence preserving summability. We put
\begin{equation}\label{eq:kernels}
 E_m(z):=\sum_{n=0}^{\infty}\frac{z^n}{m(n)}\quad\text{and}\quad e_m(z):=\sum_{n=0}^{\infty}m(n)z^n.
\end{equation}
Then $E_m(z)$ and $e_m(z)$ are called \emph{kernel functions associated with a sequence $m$}.
\end{Def}

\begin{Rem}
 The kernel functions $E_m(z)$ and $e_m(z)$ are similar in spirit to kernel functions $E(z)$ and $e(z)$ associated with a given sequence of moments $\tilde{m}$ of order $1/k>0$ and they are used in Balser's theory of moment summability (see \cite[Section 5.5]{B2}). In particular $E(z)$ satisfies the same formula as $E_m(z)$ in (\ref{eq:kernels}), if we replace the sequence preserving summability $m$ by the sequence of moments $\tilde{m}$. Nevertheless, these kernels have completely different properties.
 For example $E(z)\in\Oo^k(\CC)$ is an entire function  and $E_m(z)\in\Oo^{>0}(\CC\setminus\RR_+)$ has a singular point on $\RR_+$.
\end{Rem}

In the next proposition we collect properties of kernel functions $E_m(z)$ and $e_m(z)$ follow directly from their definitions.
\begin{Prop}\label{prop:3}
 Kernel functions $E_m(z)$ and $e_m(z)$ associated with a sequence preserving summability $m$ have the following properties.
 \begin{enumerate}
 \item[(a)] $E_m(0)=1$ and $e_m(0)=1$.
  \item[(b)] $E_m(z),e_m(z)\in\Oo(D)$. More precisely $E_m(z)\in\Oo(D_a)$ and $e_m(z)\in\Oo(D_{1/A})$, where $a,A>0$ are given by (\ref{eq:aA}).
  \item[(c)] $E_m(z),e_m(z)\in\Oo^{>0}(\CC\setminus\RR_+)$.
  \item[(d)] $E_m(z)$ and $e_m(z)$ have singular points on $\RR_+$.
\end{enumerate}
\end{Prop}
\begin{proof}
Properties (a)--(c) follow directly from Definition \ref{df:kernels}, inequality (\ref{eq:aA}) and Theorem \ref{th:2}. To prove (d) observe that if $E_m(z)$ or $e_m(z)$ has no singular points on $\RR_+$ then by (c) this kernel is an entire function, which contradicts (\ref{eq:aA}). 
\end{proof}

\begin{Ex}\label{ex:2}
 Below you find a few examples of kernel functions.
 \begin{enumerate}
  \item If $m={\bf a}=(a^n)_{n\geq 0}$ for some $a>0$ then
  $$
  E_m(z)=\frac{a}{a-z}\qquad\text{and}\qquad e_m(z)=\frac{1}{1-az}.
  $$
  In particular $E_m(z)\in\Oo(D_a\cup(\CC\setminus\RR_+))$ and $e_m(z)\in\Oo(D_{1/a}\cup(\CC\setminus\RR_+))$. Moreover, for every $\varepsilon>0$ both kernels $E_m(z)$ and $e_m(z)$ are bounded in $\CC\setminus S(0,\varepsilon)$.
  \item Let $m=([n]_q!)_{n\geq 0}$ for some $q\in[0,1)$. Then by properties of $q$-exponential function $\exp_q(z)$ and by Proposition \ref{prop:1} ($q$-binomial theorem) we get
  \begin{multline*}
  E_m(z)=\exp_q(z)=\sum_{n=0}^{\infty}\frac{z^n}{[n]_q!}=\sum_{n=0}^{\infty}\frac{(1-q)^n}{(q;q)_n}z^n=
 {}_{1}\phi_0\left(\begin{array}{c}
                       0\\
                       -
                     \end{array}
;q,(1-q)z\right)=\prod_{n=0}^{\infty}\frac{1}{1-(1-q)q^nz}.
  \end{multline*}
Hence $E_m(z)$ is a meromorphic function on $\CC$ with simple poles situated on the positive real halfline at the points 
\begin{equation*}
z=z_n=\frac{q^{-n}}{1-q}\qquad\text{for}\qquad n\in\NN_0.
\end{equation*}
In particular $E_m(z)$ is holomorphic on $D_{1/(1-q)}\cup(\CC\setminus\RR_+)$ and for every $\varepsilon>0$ the function $E_m(z)$ is bounded in $\CC\setminus S(0,\varepsilon)$.

Analogously, by Proposition \ref{prop:2} (Heine's transformation formula) we get (see also the proof of \cite[Theorem 4]{L-M})
  \begin{multline*}
e_m(z)=\sum_{n=0}^{\infty}[n]_q! z^n= \sum_{n=0}^{\infty}\frac{(q;q)_n}{(1-q)^n}z^n={}_{2}\phi_1\left(\begin{array}{c}
                       q,q\\
                       0
                     \end{array}
;q,\frac{z}{1-q}\right)\\
=\frac{\big(q,q\frac{z}{1-q};q\big)_{\infty}}{\big(\frac{z}{1-q};q\big)_{\infty}}
 \sum_{n=0}^{\infty}\frac{\big(\frac{z}{1-q};q\big)_n}{\big(q\frac{z}{1-q},q;q\big)_n}q^n=\sum_{n=0}^{\infty}\frac{1}{1-\frac{z}{1-q}q^n}\frac{(q;q)_{\infty}}{(q;q)_n}q^n.
  \end{multline*}
 \end{enumerate}
 Observe that $(q;q)_{\infty}<(q;q)_n$, $\sum_{n=0}^{\infty}q^n=\frac{1}{1-q}<\infty$ and for every $\varepsilon>0$ there exists $M<\infty$ such that
 \begin{equation*}
  \left|\frac{1}{1-\frac{z}{1-q}q^n}\right|<M\quad\text{for every}\quad n\in\NN_0\quad\text{and}\quad z\in\CC\setminus S(0,\varepsilon).
 \end{equation*}
Hence we conclude that the function $e_m(z)$ is holomorphic on $D_{(1-q)}\cup(\CC\setminus\RR_+)$ and bounded on $\CC\setminus S(0,\varepsilon)$. Moreover, $e_m(z)$ is a meromorphic function on $\CC$  with simple poles at $z=z_n=(1-q)q^{-n}$ for $n\in\NN_0$.
\end{Ex}

Kernel functions $E_m(z)$ and $e_m(z)$ allow us to get integral representations of $m$-Borel transforms of holomorphic functions.
\begin{Prop}
\label{prop:integral}
  If $u(t)=\sum_{n=0}^{\infty}u_nt^n\in\Oo(D_r,\EE)$ for some $r>0$ and $m$ satisfies (\ref{eq:aA})  then $\hat{\Bo}_mu(t)\in\Oo(D_{ar},\EE)$ and  $\hat{\Bo}_{m^{-1}}u(t)\in\Oo(D_{r/A},\EE)$. Moreover, if additionally $m$ is a sequence preserving summability then for every $\varepsilon\in(0,r)$
  these Borel transforms have integral representations
 \begin{equation}
 \label{eq:integral_1}
  \hat{\Bo}_mu(t)=\frac{1}{2\pi i}\oint_{|\zeta|=\varepsilon}\frac{u(\zeta)}{\zeta}E_m\left(\frac{t}{\zeta}\right)\,d\zeta=:\Bo_mu(t)\quad\text{for}\quad t\in D_{a\varepsilon}
 \end{equation}
and 
\begin{equation}
\label{eq:integral_2}
  \hat{\Bo}_{m^{-1}}u(t)=\frac{1}{2\pi i}\oint_{|\zeta|=\varepsilon}\frac{u(\zeta)}{\zeta}e_m\left(\frac{t}{\zeta}\right)\,d\zeta=:\Bo_{m^{-1}}u(t)\quad\text{for}\quad t\in D_{\varepsilon/A}.
 \end{equation}
\end{Prop}
\begin{proof}
 If $u(t)\in\Oo(D_r,\EE)$ then by (\ref{eq:aA}) we see that
$\hat{\Bo}_mu(t)\in\Oo(D_{ar},\EE)$ and $\hat{\Bo}_{m^{-1}}u(t)\in\Oo(D_{r/A},\EE)$. Fix $\varepsilon\in(0,r)$. By the Taylor formula for the function $u$ and by the Cauchy integral formula we get for $t\in D_{a\varepsilon}$
\begin{multline*}
 \hat{\Bo}_mu(t)=\sum_{n=0}^{\infty}\frac{u_n}{m(n)}t^n=\sum_{n=0}^{\infty}\frac{u^{(n)}(0)}{n!m(n)}t^n=\sum_{n=0}^{\infty}\frac{n!}{2\pi i}\oint_{|\zeta|=\varepsilon}\frac{u(\zeta)}{\zeta^{n+1}n!m(n)}t^n\,d\zeta\\
 =\frac{1}{2\pi i}\oint_{|\zeta|=\varepsilon}\frac{u(\zeta)}{\zeta}\sum_{n=0}^{\infty}\frac{t^n}{\zeta^n m(n)}\,d\zeta=\frac{1}{2\pi i}\oint_{|\zeta|=\varepsilon}\frac{u(\zeta)}{\zeta}E_m\left(\frac{t}{\zeta}\right)\,d\zeta.
\end{multline*}
Analogously we show the integral representation of $\hat{\Bo}_{m^{-1}}u(t)$.
\end{proof}

\begin{Prop}\label{prop:5}
 Let $m$ be a sequence preserving summability, $d\in\RR$ and $k>0$. If $u\in \Oo^k(\hat{S}_d)$ then also the functions ${\Bo}_mu$ and ${\Bo}_{m^{-1}}u$ belong to the same space $\Oo^k(\hat{S}_d)$ and for every $t\in\hat{S}_d$ they have integral representations
 \begin{equation}
 \label{eq:deform}
  {\Bo}_mu(t)=\frac{1}{2\pi i}\oint_{\gamma_{t,d}}\frac{u(\zeta)}{\zeta}E_m\left(\frac{t}{\zeta}\right)\,d\zeta\quad\text{and}\quad {\Bo}_{m^{-1}}u(t)=\frac{1}{2\pi i}\oint_{\tilde{\gamma}_{t,d}}\frac{u(\zeta)}{\zeta}e_m\left(\frac{t}{\zeta}\right)\,d\zeta,
 \end{equation}
where the contours $\gamma_{t,d}:=\partial((\hat{S}_d)_{-\delta}\cap D_{|t|/a+1})$ and $\tilde{\gamma}_{t,d}:= \partial((\hat{S}_d)_{-\delta}\cap D_{A|t|+1})$ are positively oriented, $(\hat{S}_d)_{-\delta}:=\{z\in\hat{S}_d\colon \dist(z,\partial\hat{S}_d)>\delta\}$ and $\delta:=2\dist(t,\partial\hat{S}_d)$.
\end{Prop}
\begin{proof}
 By Proposition \ref{prop:integral} ${\Bo}_mu$ has the integral representation given by (\ref{eq:integral_1}). We fix $t\in\hat{S}_d$.  By the Cauchy integral theorem, we can deform the contour of integration $|\zeta|=\varepsilon$ in (\ref{eq:integral_1}) to $\gamma_{t,d}$. Indeed, it is possible since $u\in \Oo^k(\hat{S}_d)$ and $t/\zeta\in D_a\cup(\CC\setminus\RR_+)$ for every $\zeta\in\gamma_t$. Hence we conclude that ${\Bo}_mu$ satisfies (\ref{eq:deform}) for every $t\in\hat{S}_d$. Next, using (\ref{eq:deform}) we estimate the growth of ${\Bo}_mu(t)$. Since $u(t)\in \Oo^k(\hat{S}_d)$ and $E_m(z)\in\Oo^{>0}(\CC\setminus\RR_+)$ we conclude that ${\Bo}_mu(t)$ has also exponential growth of order $k$ as $t\to\infty$, $t\in\hat{S}_d$. Analogously we get the similar result for ${\Bo}_{m^{-1}}u(t)$.
\end{proof}

Immediately from the above proposition, Theorem \ref{th:2} and Definition \ref{df:preserve} we get a new characterization of sequences preserving summability.
\begin{Prop}
 A sequence $m=(m(n))_{n\geq 0}$ preserves summability if and only if for any $k>0$, $d\in\RR$ and for every $\varphi\in\Oo(D)$ there exists a disc-sector $\hat{S}_d$ such that the following equivalence holds:
  \begin{equation*}
  \varphi\in\Oo^{k}(\hat{S}_d)\quad \textrm{if and only if}\quad {\Bo}_{m}\varphi\in\Oo^{k}(\hat{S}_d).
  \end{equation*}
\end{Prop}
\bigskip\par
 
We will introduce the following convolution of kernels in the similar spirit to \cite[Proposition 4.15]{J-K-L-S}.  
\begin{Def}
 Let $m_1$ and $m_2$ be sequences preserving summability and satisfying (\ref{eq:aA}) with $A>0$. Then the \emph{convolution of kernels $e_{m_1}$ and $e_{m_2}$} is a holomorphic function on $D_{1/A^2}$ defined by
 \begin{equation}
 \label{eq:convolution}
  e_{m_1}*e_{m_2}(z):=\frac{1}{2\pi i}\oint_{|\zeta|=\varepsilon}e_{m_1}(z/\zeta) e_{m_2}(\zeta)\frac{d\zeta}{\zeta}\quad\text{for every}\quad z\in D_{\varepsilon/A},
 \end{equation}
where $\varepsilon<1/A$ is fixed.
\end{Def}
\begin{Rem}
 Since $E_m(z)=e_{m^{-1}}(z)$, the formula (\ref{eq:convolution}) describes also convolutions $E_{m_1}*e_{m_2}(z)$, $e_{m_1}*E_{m_2}(z)$ and $E_{m_1}*E_{m_2}(z)$.
\end{Rem}

In the next proposition we describe fundamental properties of the convolutions of kernels
\begin{Prop}
If $m_1$ and $m_2$ are sequences preserving summability which satisfy                                                                (\ref{eq:aA}) for some $A>0$ then the convolution of kernels $e_{m_1}$ and $e_{m_2}$ is analytically continued to $D_{1/A^2}\cup(\CC\setminus\RR_+)$ and satisfies
 \begin{equation}
 \label{eq:conv}
  e_{m_1}*e_{m_2}(z)=e_{m_1m_2}(z)=e_{m_2}*e_{m_1}(z)\quad \text{for every}\quad z\in D_{1/A^2}.
 \end{equation}
\end{Prop}
\begin{proof}
Fix $z\in D_{1/A^2}$ and $\varepsilon\in(A|z|,1/A)$. Since by Proposition \ref{prop:3} the series $e_{m_1}(z/\zeta)$ and $e_{m_2}(\zeta)$ are convergent uniformly on $|\zeta|=\varepsilon$,
we can change the order of integration and summation
\begin{multline*}
 e_{m_1}*e_{m_2}(z):=\frac{1}{2\pi i}\oint_{|\zeta|=\varepsilon}e_{m_1}(z/\zeta) e_{m_2}(\zeta)\frac{d\zeta}{\zeta}\\
 =\frac{1}{2\pi i}\oint_{|\zeta|=\varepsilon}\left(\sum_{n=0}^{\infty}m_1(n)(z/\zeta)^n\right)\left(\sum_{k=0}^{\infty}m_2(k)\zeta^k\right)\frac{d\zeta}{\zeta}\\
 =\sum_{n=0}^{\infty}m_1(n)\left(\sum_{k=0}^{\infty}\frac{m_2(k)}{2\pi i}\oint_{|\zeta|=\varepsilon}\frac{d\zeta}{\zeta^{n-k+1}}\right)z^n.
\end{multline*}
Moreover, by the Cauchy integral formula we get 
\begin{equation*}
 \sum_{n=0}^{\infty}m_1(n)\left(\sum_{k=0}^{\infty}\frac{m_2(k)}{2\pi i}\oint_{|\zeta|=\varepsilon}\frac{d\zeta}{\zeta^{n-k+1}}\right)z^n=\sum_{n=0}^{\infty}m_1(n)m_2(n)z^n=e_{m_1m_2}(z)
\end{equation*}
and the first equality in (\ref{eq:conv}) holds.
To show the second equality in (\ref{eq:conv}) it is sufficient to observe that $m_1m_2=(m_1(n)m_2(n))_{n\geq 0}=(m_2(n)m_1(n))_{n\geq 0}=m_2m_1$ and to repeat the first part of the proof with $m_1$ replaced by $m_2$ and $m_2$ replaced by $m_1$.
\end{proof}
\begin{Rem} By \cite[page 84]{H} the convolution of holomorphic functions $\varphi,\psi\in\Oo(D)$ given by
\begin{equation*}
 \frac{1}{2\pi i}\oint_{|\zeta|=\varepsilon}\varphi(z/\zeta) \psi(\zeta)\frac{d\zeta}{\zeta}
\end{equation*}
is the integral representation of the \emph{Hadamard product} of these functions, which is defined as
 \begin{equation}
  \label{eq:hadamard}
  \varphi*\psi(z):=\sum_{n=0}^{\infty}\frac{\varphi^{(n)}(0)\psi^{(n)}(0)}{(n!)^2}z^n.
 \end{equation}
 Applying (\ref{eq:hadamard}) to $\varphi(z)=e_{m_1}(z)$ and $\psi(z)=e_{m_2}(z)$ we get directly
 \begin{equation*}
 e_{m_1}*e_{m_2}(z)=\left(\sum_{n=0}^{\infty}m_1(n)z^n\right)*\left(\sum_{n=0}^{\infty}m_2(n)z^n\right)=\sum_{n=0}^{\infty}m_1(n)m_2(n)z^n=e_{m_1m_2}(z),
 \end{equation*}
which gives the another proof of (\ref{eq:conv}).
\end{Rem}

\section{Sequences of generalized moments}\label{sec:7}
Let $\tilde{m}=(\tilde{m}(n))_{n\geq0}$ be a sequence of moments of order $s>0$ in Balser's sense (see \cite[page 86]{B2}). Then there exist $a,A>0$ such that
\begin{equation*}
 a^n(n!)^s\leq \tilde{m}(n)\leq A^n(n!)^s\quad\text{for every}\quad n\in\NN_0.
\end{equation*}
The standard example of such sequence of moments of order $s>0$ is given by $\Gamma_s:=(\Gamma(1+ns))_{n\geq0}$. Then we may write
\begin{equation*}
 \tilde{m}(n)=\frac{\tilde{m}(n)}{\Gamma_s(n)}\Gamma_s(n)=m(n)\Gamma_s(n),
\end{equation*}
where $m=(m(n))_{n\geq0}=\left(\frac{\tilde{m}(n)}{\Gamma_s(n)}\right)_{n\geq 0}$ is a sequence of moments of order zero (see \cite[Definition~4]{Mic8} for a definition of a moment function of order zero). By example \ref{ex:1} $m$ preserves summability. 
It suggests to extend the sequences of moments of order $s>0$ to sequences $m\Gamma_s$, where $m$ preserves summability.
\begin{Def}\label{df:generalized}
 If $m$ is a sequence preserving summability and $s\geq 0$ then $\tilde{m}=m\Gamma_s$ is called a \emph{sequence of generalized moments of order $s$}.
\end{Def}

\begin{Rem}
Observe that the set of sequences of generalized moments contains the set of sequences preserving summability and the set of sequences of moments. Moreover, since $\Gamma_0=(\Gamma(1+n0))_{n\geq 0}=(1)_{n\geq0}=\mathbf{1}$ we see that the set of sequences of generalized moments of order $0$ coincides with the set of sequences preserving summability. 
\end{Rem}

\begin{Rem}
 Directly be the definition and by properties of sequences of moments and sequences preserving summability we conclude that the set of sequences of generalized moments forms a semigroup with multiplication and with identity $\mathbf{1}=(1)_{n\geq 0}$.
\end{Rem}

\begin{Rem}\label{re:summable}
 Directly by Definitions \ref{df:borel}, \ref{df:summable}, \ref{df:preserve} and \ref{df:generalized} we see that for every sequence $\tilde{m}$ of generalized moments of order $1/k$ the following equivalence holds: $\hat{u}\in\EE[[t]]$ is $k$-summable in a direction $d$ if and only if there exists a disc-sector $\hat{S}_d$ in a direction $d$ such that $\hat{\Bo}_{\tilde{m}}\hat{u}\in\Oo^k(\hat{S}_d,\EE)$.
\end{Rem}

\section{Extension of Laplace and Borel integral operators}\label{sec:8}
In this section we extend the definitions of Laplace and Borel integral operators $\La_{\Gamma_{1/k},d}$ and $\Bo_{\Gamma_{1/k},d}$ given for the sequence $\Gamma_{1/k}$ to any sequence of generalized moments of order $1/k$.
We use the operators $\La_{\Gamma_{1/k},d}$ and $\Bo_{\Gamma_{1/k},d}$ constructed in Definition \ref{def:LB_operators} and the operators $\Bo_m$ and $\Bo_{m^{-1}}$ given by (\ref{eq:deform}) to construct integral operators connected with sequences of generalized moments.
\begin{Def}
 If $\tilde{m}=m\Gamma_s$ is a sequence of generalized moments of order $s>0$ and $k=1/s$ then we can define \emph{Laplace-like} $\La_{\tilde{m},d}$ and \emph{Borel-like} $\Bo_{\tilde{m},d}$ \emph{integral operators} connected with the formal operators $\hat{\Bo}_{\tilde{m}^{-1}}$ and $\hat{\Bo}_{\tilde{m}}$ as
 \begin{equation}\label{eq:like}
  \La_{\tilde{m},d}:=\La_{\Gamma_{1/k},d}\circ\Bo_{m^{-1}}\quad\text{and}\quad \Bo_{\tilde{m},d}:=\Bo_m\circ\Bo_{\Gamma_{1/k},d}.
 \end{equation}
\end{Def}

\begin{Rem}\label{re:laplace}
 Operators $\La_{\tilde{m},d}$ and $\Bo_{\tilde{m},d}$ are well defined. Indeed by Proposition \ref{prop:5}, if $v\in\Oo^k(\hat{S}_d)$ then $\Bo_{m^{-1}}v\in\Oo^k(\hat{S}_d)$  and if $\Bo_{\Gamma_{1/k},d}u\in\Oo^k(\hat{S}_d)$ then also $\Bo_m(\Bo_{\Gamma_{1/k},d}u)\in\Oo^k(\hat{S}_d)$. 
 
 Moreover, if $\hat{S}_d=\hat{S}_d(\varepsilon)$ for some $\varepsilon>0$ and $u\in\Oo^k(\hat{S}_d)$ then by \cite[Section 5.1]{B2} and by Proposition \ref{prop:5} we conclude that $\La_{\tilde{m},d}u\in\Oo(G(d,\pi/k+\varepsilon))$, where $G(d,\pi/k+\varepsilon)$ denotes a sectorial region of opening $\pi/k+\varepsilon$ and bisecting direction $d$. Conversely, if $v\in\Oo(G(d,\pi/k+\varepsilon))$ then by \cite[Section 5.2]{B2} and by Proposition \ref{prop:5} we see that $\Bo_{\tilde{m},d}v\in\Oo^k(\hat{S}_d(\varepsilon))$.
\end{Rem}

\begin{Rem}\label{re:sum}
 Let $k>0$ and $d\in\RR$. If $\hat{u}\in\EE\{t\}_{k,d}$ and $\tilde{m}=m\Gamma_{1/k}$ is a sequence of generalized moments of order $1/k$ then $k$-sum of $\hat{u}$ in the direction $d$ may be written as
 \begin{equation*}
  u^d(t)=\Ss_{k,d}\hat{u}(t)=\La_{\tilde{m},d}\circ \hat{\Bo}_{\tilde{m}}\hat{u}(t).
 \end{equation*}
Indeed, since $v(t)=\hat{\Bo}_{\Gamma_{1/k}}\hat{u}(t)\in\Oo^k(\hat{S}_d,\EE)$, by Propositions \ref{prop:integral} and \ref{prop:5} we see that also 
$\hat{\Bo}_mv(t)=\Bo_mv(t)\in\Oo^k(\hat{S}_d,\EE)$. Therefore
\begin{multline*}
 \La_{\tilde{m},d}\circ \hat{\Bo}_{\tilde{m}}\hat{u}(t)=\La_{\Gamma_{1/k},d}\circ\Bo_{m^{-1}}\circ\hat{\Bo}_m\circ \hat{\Bo}_{\Gamma_{1/k}}\hat{u}(t)=\La_{\Gamma_{1/k},d}\circ\Bo_{m^{-1}}\circ\Bo_m\circ \hat{\Bo}_{\Gamma_{1/k}}\hat{u}(t)\\
 =\La_{\Gamma_{1/k},d}\circ \hat{\Bo}_{\Gamma_{1/k}}\hat{u}(t)=\Ss_{k,d}\hat{u}(t).
\end{multline*}

\end{Rem}

\begin{Rem} Using Balser's theory of moment summability \cite[Chapter 6.5]{B2} for given $k>0$ and $d\in\RR$ we can replace the Borel transform $\Bo_{\Gamma_{1/k},d}$ by the \emph{modified Borel transform $\tilde{\Bo}_{\Gamma_{1/k},d}$ of order $k$ in a direction $d$},
which is defined in the following way.
If $u\in\Oo(S_d(\frac{\pi}{k}+\varepsilon,R),\EE)$ for some $\varepsilon>0$ and $R<\infty$ then  $\tilde{\Bo}_{\Gamma_{1/k},d}$ is defined by
  \[
   (\tilde{\Bo}_{\Gamma_{1/k},d}u)(s):=-\frac{1}{2\pi i}\int_{\gamma_k(d)}\mathbf{E}_{1/k}(s/t)u(t)\frac{dt}{t},
  \]
where $\gamma_k(d)$ is the same path as in Definition  \ref{def:LB_operators} and $\mathbf{E}_{1/k}(z)=\sum_{n=0}^{\infty}\frac{z^n}{\Gamma(1+n/k)}$ denotes the Mittag-Leffler function of index $1/k$. Hence we may also replace the Borel-like integral operator $\Bo_{\tilde{m},d}$ defined in (\ref{eq:like}), by the \emph{modified Borel-like operator}
$\tilde{\Bo}_{\tilde{m},d}$ defined as
\begin{equation*}
\tilde{\Bo}_{\tilde{m},d}:=\Bo_m\circ\tilde{\Bo}_{\Gamma_{1/k},d}.
\end{equation*}

\end{Rem}

\section{Moment differential operators for sequences of generalized moments}\label{sec:9}

The notion of $m$-moment differentiation was introduced by Balser and Yoshino~\cite{B-Y} in the case when $m$ is a moment sequence associated with a pair of kernel functions (see \cite[Section 5.5]{B2}) and then extended in~\cite{I-M} to any sequence $m=(m(n))_{n\geq 0}$ of positive numbers with $m(0)=1$. 
\begin{Def}
 For a given sequence $m=(m(n))_{n\geq 0}$ of positive numbers with $m(0)=1$, an operator $\partial_{m,t}\colon\EE[[t]] \to \EE[[t]]$ defined by
 \begin{equation*}
  \partial_{m,t}\big(\sum_{n=0}^{\infty}u_nt^n\big):=\sum_{n=0}^{\infty}\frac{m(n+1)}{m(n)}u_{n+1}t^n
 \end{equation*}
is called an \emph{$m$-moment differential operator}.
\end{Def}

\begin{Ex}\phantom{a}\\[-2ex]
\begin{enumerate}
  \item In the most important case $\Gamma_1=(\Gamma(1+n))_{n\geq 0}=(n!)_{n\geq 0}$, the operator $\partial_{\Gamma_1,t}$ is the $\Gamma_1$-moment differential operator, which coincides with the usual derivative $\partial_t$.
  \item More generally, if $s>0$ and $\Gamma_s=(\Gamma(1+sn))_{n\geq 0}$ then the operator $\partial_{\Gamma_s,t}$ satisfies
     $(\partial_{\Gamma_s,t}\hat{u})(t^s)=\partial^s_t(\hat{u}(t^s))$,
     where $\partial^s_t$ denotes the Caputo fractional derivative of order $s$
     defined by
       $$
     \partial^{s}_{t}\Big(\sum_{j=0}^{\infty}\frac{u_{j}}{\Gamma(1+sj)}t^{sj}\Big):=
     \sum_{j=0}^{\infty}\frac{u_{j+1}}{\Gamma(1+sj)}t^{sj}.$$
 \item If $m=([n]_q!)_{n\geq 0}$ then the operator $\partial_{m,t}$ coincides with the Jackson $q$-derivative $$D_{q,t} u(t)=\frac{u(qt)-u(t)}{qt -t}.$$
 \end{enumerate}

\end{Ex}

By the direct calculation we receive the following commutation formula between $m_1$-Borel operator and $m_2$-moment differentiation.
\begin{Prop}[see {\cite[Proposition 7]{Mic8}}]
\label{prop:commutate}
 Let $m_1=(m_1(n))_{n\geq 0}$ and $m_2=(m_2(n))_{n\geq 0}$ be sequences of positive numbers. Then the operators $\hat{\Bo}_{m_1},\partial_{m_2,t}\colon\EE[[t]]\to\EE[[t]]$ commute in a such way that
 \begin{equation*}
  \hat{\Bo}_{m_1}\partial_{m_2,t}=\partial_{m_1m_2,t}\hat{\Bo}_{m_1}.
 \end{equation*}
\end{Prop}

We have the following integral representation of $\Gamma_s$-moment derivative of a holomorphic function in a complex neighbourhood of the origin.
\begin{Prop}[{\cite[Proposition 10]{Mic7}}]
 \label{prop:integral_rep}
 Let  $s>0$, $k=1/s$ and $\varphi\in\Oo(D_r)$ for some $r>0$.
 Then for every $|z|<\varepsilon<r$ and $n\in\NN_0$ we have
 \begin{equation}
 \partial_{\Gamma_s,z}^n\varphi(z)=\frac{1}{2\pi i}\oint_{|w|=\varepsilon}\varphi(w)\int_0^{\infty(e^{i\theta})}\zeta^n\mathbf{E}_s(z\zeta)k(w\zeta)^{k-1}e^{-(w\zeta)^k}\,d\zeta\, dw,
\end{equation}
where $\theta\in(-\arg w - \frac{\pi}{2k},-\arg w + \frac{\pi}{2k})$ and $\mathbf{E}_s$ denotes the Mittag-Leffler function of index $s$, i.e. $\mathbf{E}_s(z)=E_{\Gamma_s}(z)=\sum_{n=0}^{\infty}\frac{z^n}{\Gamma(1+sn)}$.
\end{Prop}

Using the above propositions we show
\begin{Thm}\label{th:3}
 We assume that $s>0$, $k=1/s$, $m$ is a sequence preserving summability satisfying (\ref{eq:aA}), $\tilde{m}=\Gamma_s m$ is a sequence of generalized moments of order $s$ and $\varphi\in\Oo(D_r)$ for some $r>0$. Then $\tilde{m}$-moment derivatives $\partial_{\tilde{m},z}^n$ of $\varphi$ have the following integral representation for every $n\in\NN_0$ and $|z|<\varepsilon<r/A$:
 \begin{equation}
  \partial_{\tilde{m},z}^n\varphi(z)
  =\frac{1}{2\pi i}\oint_{|w|=\varepsilon}(\Bo_{m^{-1}}\varphi)(w)\int_0^{\infty(e^{i\theta})}\zeta^n E_{\tilde{m}}(z\zeta)k(w\zeta)^{k-1}e^{-(w\zeta)^k}\, d\zeta\, dw,
 \end{equation}
 where $\theta\in(-\arg w - \frac{\pi}{2k},-\arg w + \frac{\pi}{2k})$ and $E_{\tilde{m}}(z)=\sum_{n=0}^{\infty}\frac{z^n}{\tilde{m}(n)}$.
\end{Thm}
\begin{proof}
Let $\varphi\in\Oo(D_r)$. By Proposition \ref{prop:integral} $\tilde{\varphi}:=\Bo_{m^{-1}}\varphi\in\Oo(D_{r/A})$ and $\varphi=\hat{\Bo}_{m}\tilde{\varphi}$. Hence by Propositions \ref{prop:commutate} and \ref{prop:integral_rep} we conclude that 
\begin{multline*}
 \partial^n_{\tilde{m},z}\varphi(z)=\partial^n_{\Gamma_s m,z}\hat{\Bo}_{m}\tilde{\varphi}(z)=\hat{\Bo}_{m}\left(\partial^n_{\Gamma_s,z}\tilde{\varphi}(z)\right)\\
 =
 \hat{\Bo}_{m}\left(\frac{1}{2\pi i}\oint_{|w|=\varepsilon}\tilde{\varphi}(w)\int_0^{\infty(e^{i\theta})}\zeta^n\mathbf{E}_s(z\zeta)k(w\zeta)^{k-1}e^{-(w\zeta)^k}\,d\zeta\, dw\right)\\
 =\frac{1}{2\pi i}\oint_{|w|=\varepsilon}\tilde{\varphi}(w)\int_0^{\infty(e^{i\theta})}\zeta^nE_{\Gamma_sm}(z\zeta)k(w\zeta)^{k-1}e^{-(w\zeta)^k}\,d\zeta\, dw\\
 =\frac{1}{2\pi i}\oint_{|w|=\varepsilon}\Bo_{m^{-1}}\varphi(w)\int_0^{\infty(e^{i\theta})}\zeta^nE_{\tilde{m}}(z\zeta)k(w\zeta)^{k-1}e^{-(w\zeta)^k}\,d\zeta\, dw,
\end{multline*}
where $\theta\in(-\arg w - \frac{\pi}{2k},-\arg w + \frac{\pi}{2k})$.
\end{proof}

\section{Summability of formal solutions of some moment-pdes}\label{sec:10}
We will study formal solutions of some moment partial differential equations for sequences of generalized moments. To this end we need
 \begin{Lem}\label{le:1}
  Fix $k>0$ and $d\in\RR$. A function $u(t,z)=\sum_{n=0}^{\infty}u_n(t)z^n$ belongs to the space $\Oo^k(\hat{S}_d\times D)$ if and only if the following conditions hold:
  \begin{enumerate}
   \item[(i)] $u_n(t)\in\Oo^k(\hat{S}_d)$ for every $n\in\NN_0$,
   \item[(ii)] there exists a positive constant $B<\infty$ such that for every sub-disc-sector $\hat{\tilde{S}}_d$ of $\hat{S}_d$ one can find positive constants $A,C<\infty$ such that
   \begin{equation}\label{eq:abc}
    |u_n(t)|\leq AB^ne^{C|t|^k}\quad\text{for every}\quad t\in\hat{\tilde{S}}_d\quad\text{and}\quad n\in\NN_0.
   \end{equation}
  \end{enumerate}
 \end{Lem}
 \begin{proof}
 ($\Longrightarrow$) 
 Since $u(t,z)\in\Oo^k(\hat{S}_d\times D)$, we see that
 \begin{equation*}
  u_n(t)=\frac{\partial_z^n u(t,z)|_{z=0}}{n!}\in\Oo^k(\hat{S}_d)\quad\text{for every}\quad n\in\NN_0.
 \end{equation*}
Moreover, by its integral representation we get
\begin{equation*}
 u_n(t)=\frac{1}{2\pi i}\oint_{|w|=\varepsilon}\frac{u(t,w)}{w^{n+1}}\,dw\quad\text{for sufficiently small}\quad \varepsilon>0.
\end{equation*}
Hence for $B=1/\varepsilon$ we get
\begin{equation*}
 |u_n(t)|\leq \frac{\sup_{w\in D_{\varepsilon}}|u(t,w)|}{\varepsilon^n}\leq \sup_{w\in D_{\varepsilon}}|u(t,w)|B^n\quad\text{for every}\quad n\in\NN_0.
\end{equation*}
Since $u(t,w)\in\Oo^k(\hat{S}_d\times D)$, for every sub-disc-sector $\hat{\tilde{S}}_d$ of $\hat{S}_d$ one can find positive constants $A,C<\infty$ such that
\begin{equation*}
 \sup_{w\in D_{\varepsilon}}|u(t,w)|\leq Ae^{C|t|^k}\quad\text{for}\quad t\in \hat{\tilde{S}}_d,
\end{equation*}
and (ii) holds.

($\Longleftarrow$) To prove the second implication, take  $B<\infty$ from (ii) and fix a sub-disc-sector $\hat{\tilde{S}}_d$ of $\hat{S}_d$. By (ii) there exist $A,C<\infty$ such that (\ref{eq:abc}) holds. Then for every $(t,z)\in\hat{\tilde{S}}_d\times D_{1/B}$ we have
  \begin{equation*}
   \left|\sum_{n=0}^{\infty}u_n(t)z^n\right|\leq \sum_{n=0}^{\infty}|u_n(t)||z|^n\leq Ae^{C|t|^k}\sum_{n=0}^{\infty}B^n|z|^n\leq \frac{Ae^{C|t|^k}}{1-B|z|}<\infty.
  \end{equation*}
It means that the series $\sum_{n=0}^{\infty}u_n(t)z^n$ is convergent on $\hat{\tilde{S}}_d\times D_{1/B}$ and it sum $u(t,z)$ belongs to the space $\Oo^k(\hat{\tilde{S}}_d\times D)$ with $D=D_{1/B}$. By the freedom of choice of a sub-disc-sector $\hat{\tilde{S}}_d$, we conclude that $u(t,z)$ belongs also to the space $\Oo^k(\hat{S}_d\times D)$.
 \end{proof}

 \begin{Lem}\label{le:2}
  Let $a>0$ and $m=(m(n))_{n\geq0}$ be a sequence of positive numbers satisfying $m(n)\geq a^n$ for every $n\in\NN_0$. We also assume that a function $u(t,z)=\sum_{n=0}^{\infty}u_n(t)z^n$ belongs to the space $\Oo^k(\hat{S}_d\times D)$ for some $k>0$ and $d\in\RR$. Then there exists 
 a complex disc $\tilde{D}$ such that $\hat{\Bo}_{m,z}u(t,z)\in\Oo^k(\hat{S}_d\times \tilde{D})$.
 \end{Lem}
\begin{proof}
 If $u(t,z)\in\Oo^k(\hat{S}_d\times D)$ then by Lemma \ref{le:1} functions $u_n(t)$ satisfy conditions (i) and (ii) from Lemma \ref{le:1}. If we put $\tilde{u}_n(t):=u_n(t)/m(n)$ for $n\in\NN_0$ then we see that also functions $\tilde{u}_n(t)$ satisfy (i) and (ii) from Lemma \ref{le:1} with $B$ replaced by $B/a$. Hence, also by Lemma \ref{le:1} we conclude that $\hat{\Bo}_{m,z}u(t,z)=\sum_{n=0}^{\infty}\tilde{u}_n(t)z^n\in\Oo^k(\hat{S}_d\times \tilde{D})$ for some complex disc $\tilde{D}$.
\end{proof}

Now we are ready to formulate our problem.
Since by Definition \ref{df:summable} $\hat{u}\in\Oo\{t\}_{k,d}$ if and only if $\hat{\Bo}_{\Gamma_{1/k},t}\hat{u}\in\Oo^k(\hat{S}_d\times D)$, directly by Lemma \ref{le:2} we get
\begin{Prop}\label{pr:sum}
 Fix $k>0$ and $d\in\RR$. If a sequence $m=(m(n))_{n\geq 0}$ satisfies (\ref{eq:aA}) then $\hat{u}(t,z)\in\Oo\{t\}_{k,d}$ if and only if $\hat{\Bo}_{m,z}\hat{u}(t,z)\in\Oo\{t\}_{k,d}$. 
\end{Prop}

We assume that $m_1$ and $m_2$ are sequences preserving summability, $s_1,s_2\geq 0$, $\tilde{m}_1=m_1\Gamma_{s_1}$ and $\tilde{m}_2=m_2\Gamma_{s_2}$ are sequences of generalized moments  and $P(\lambda,\zeta)$ is a general polynomial of two variables of order $p$ with respect to $\lambda$ and $\varphi_j(z)\in\Oo(D)$ for $j=0,\dots,p-1$.

We will investigate the relationship between the solution $\hat{u}\in\Oo[[t]]$ of the Cauchy problem
\begin{equation}
 \label{eq:CP_u}
    \left\{
    \begin{array}{l}
     P(\partial_{\tilde{m}_1,t},\partial_{\tilde{m}_2,z})u=0\\
     \partial^j_{\tilde{m}_1,t}u(0,z)=\varphi_j(z),\quad j=0,\dots,p-1
    \end{array}
    \right.
 \end{equation}
and the solution $\hat{v}\in\Oo[[t]]$ of the similar initial value problem of the form
\begin{equation}
\label{eq:CP_tilde_u}
    \left\{
    \begin{array}{l}
     P(\partial_{\Gamma_{s_1},t},\partial_{\Gamma_{s_2},z})v=0\\
     \partial^j_{\Gamma_{s_1},t}v(0,z)=\Bo^{-1}_{m_2,z}\varphi_j(z),\quad j=0,\dots,p-1.
    \end{array}
    \right.
 \end{equation}

\begin{Thm}\label{th:4} The formal power series
 $\hat{u}\in\Oo[[t]]$ is a formal solution of (\ref{eq:CP_u}) if and only if $\hat{v}=\hat{\Bo}_{m^{-1}_1,t}\hat{\Bo}_{m^{-1}_2,z}\hat{u}\in\Oo[[t]]$ is a formal solution of (\ref{eq:CP_tilde_u}). Moreover,  the equivalence $\hat{u}\in\Oo\{t\}_{k,d}$ if and only if $\hat{v}\in\Oo\{t\}_{k,d}$ holds for every $k>0$ and $d\in\RR$.
\end{Thm}
\begin{proof}
 Let $\hat{u}\in\Oo[[t]]$ be a formal solution of (\ref{eq:CP_u}). Applying Borel transforms $\hat{\Bo}_{m^{-1}_1,t}\hat{\Bo}_{m^{-1}_2,z}$ to both sides of (\ref{eq:CP_u}) and using Proposition \ref{prop:commutate} we get 
 \begin{equation*}
  \hat{\Bo}_{m^{-1}_1,t}\hat{\Bo}_{m^{-1}_2,z}P(\partial_{m_1\Gamma_{s_1},t},\partial_{m_2\Gamma_{s_2},z})u=P(\partial_{\Gamma_{s_1},t},\partial_{\Gamma_{s_2},z})\hat{\Bo}_{m^{-1}_1,t}\hat{\Bo}_{m^{-1}_2,z}u=P(\partial_{\Gamma_{s_1},t},\partial_{\Gamma_{s_2},z})v=0
 \end{equation*}
and
\begin{equation*}
 \hat{\Bo}_{m^{-1}_1,t}\hat{\Bo}_{m^{-1}_2,z}\partial^j_{m_1\Gamma_{s_1},t}u(0,z)=\partial^j_{\Gamma_{s_1},t}\hat{\Bo}_{m^{-1}_1,t}\hat{\Bo}_{m^{-1}_2,z}u(0,z)=\partial^j_{\Gamma_{s_1},t} v(0,z)=\Bo_{m^{-1}_2,z}\varphi_j(z).
\end{equation*}
So $\hat{v}$ is a formal solution of (\ref{eq:CP_tilde_u}). Analogously, applying Borel transforms $\hat{\Bo}_{m_1,t}\hat{\Bo}_{m_2,z}$ to both sides of (\ref{eq:CP_tilde_u}) we conclude that if $\hat{v}$ is a formal solution of (\ref{eq:CP_tilde_u}) then $\hat{u}=\hat{\Bo}_{m_1,t}\hat{\Bo}_{m_2,z}\hat{v}$ is a formal solution of (\ref{eq:CP_u}).

To prove the second part of the theorem, first observe that since $m_1$ is a sequence preserving summability, then $\hat{u}\in\Oo\{t\}_{k,d}$ if and only if $\hat{\Bo}_{m_1,t}\hat{u}\in\Oo\{t\}_{k,d}$.
To finish the proof it is sufficient to show that
if $m_2$ is a sequence preserving summability, then $\hat{u}\in\Oo\{t\}_{k,d}$ if and only if $\hat{\Bo}_{m_2,z}\hat{u}\in\Oo\{t\}_{k,d}$. But it follows directly from Proposition \ref{pr:sum}, since by Theorem \ref{th:1} a sequence preserving summability $m_2$ satisfies (\ref{eq:aA}).
\end{proof}
\begin{Rem} By the above proof we see that in Theorem \ref{th:4} one can replace the assumption that a sequence $m_2$ preserves summability by a weaker one that $m_2$ satisfies (\ref{eq:aA}).
 \end{Rem}
  
\section{Singular directions for Borel sums in terms of hyperfunctions}\label{sec:11}
Now we recall the concept of the Stokes phenomenon for summable (see \cite[Definition 7]{M-P}) and multisummable (see \cite[Section 5]{M-T}) formal power series
$\hat{u}\in\EE[[t]]$.
\begin{Def}
\label{df:stokes}
If $\hat{u}\in\EE\{t\}_k$ has singular directions
$d_1,\dots,d_n$ then every half-line
${L}_{d_{j}}=\{t\in\widetilde{\CC}\colon \arg t=d_{j}\}$ ($j=1,\dots,n$)
is called a \emph{Stokes line} for $\hat{u}$.

We assume that $d_j^+$ (resp. $d_j^-$) denotes a direction close to $d_j$ and greater (resp. less) than $d_j$ and $u^{d_j^{\pm}}:=\Ss_{k,d_j^{\pm}}\hat{u}$.
Then the difference $J_{L_{d_j},k}\hat{u}:= u^{d_j^+}-u^{d_j^-}$ is called a \emph{jump for $\hat{u}$ across
the Stokes line $L_{d_j}$} or a \emph{jump for $\hat{u}$ across
the singular direction $d_j$}.
\end{Def}

%

We will describe jumps across the Stokes lines in terms of hyperfunctions. The similar approach to the Stokes phenomenon one can find
in \cite{Im, Mal3, M-T, S-S}. For more information about the theory of hyperfunctions we refer the reader to \cite{Kaneko}.

We will consider the space 
$$\mathcal{H}^k(L_d):=\Oo^k(D\cup (S_d\setminus L_d))\Big/\Oo^k(\hat{S}_d)$$
of Laplace type hyperfunctions supported by $L_d$ with exponential growth of order $k$ (see also \cite{S-Z}). It means that every hyperfunction
$G\in\mathcal{H}^k(L_d)$ may be written as
\[
G(s)=[g(s)]_{d}=\{g(s)+h(s)\colon h(s)\in\Oo^{k}(\hat{S}_d)\}
\]
for some defining function $g(s)\in\Oo^k(D\cup (S_d\setminus L_d))$. 

Let $\gamma_{d}$ be a path consisting of the half-lines from 
$e^{id^-}\infty$ to $0$ and from $0$ to $e^{id^+}\infty$, i.e.
$\gamma_{d}=-\gamma_{d^-}+\gamma_{d^+}$ with $\gamma_{d^{\pm}}={L}_{d^{\pm}}$.
By the K\"othe type theorem \cite{Kot} (see also \cite[Theorem 5.1]{S-Z}) one can treat the hyperfunction $G(s)=[g(s)]_d$ as the analytic functional defined by
\begin{gather}
\label{eq:kothe}
G(s)[\varphi(s)]:=\int_{\gamma_d}g(s)\varphi(s)\,ds,
\end{gather}
for small functions $\varphi\in\Oo(\widehat{S}_d)$ exponentially decay of order $k$ at infinity  such that
the function $s\mapsto g(s)\varphi(s)$ is also exponentially decay of order $k$ on $D\cup (S_d\setminus L_d)$.

For given sequence of generalized moments we can describe the jumps across the Stokes lines in terms of hyperfunctions in th following way
\begin{Prop}\label{prop:stokes}
 Let $k>0$ and $d\in\RR$. We assume that $\hat{f}\in\CC[[t]]$ is $k$-summable, $\tilde{m}=\Gamma_{1/k}m$ is a sequence of generalized moments of order $1/k$ and $d$ is a singular direction. Then
\begin{equation*}
 J_{L_d}\hat{f}(t)=f^{d^+}(t)-f^{d^-}(t)=\big(\mathcal{L}_{\tilde{m},d^+}-\mathcal{L}_{\tilde{m},d^-}\big)\hat{\Bo}_{\tilde{m}}\hat{f}(t)=G_0(s)\left[\frac{k(s/t)^ke^{-(s/t)^k}}{s}\right]
\end{equation*}
for sufficiently small $r>0$ and $t\in S_{d}(\frac{\pi}{k}, r)$, where
the hyperfunction $G_0(s)$ is defined by
\begin{equation*}
 G_0(s)=\left[\Bo_{m^{-1}}(\hat{\Bo}_{\tilde{m}}\hat{f}(s))\right]_d=
 \left[\frac{1}{2\pi i}\oint_{\tilde{\gamma}_{s,d}}\frac{\hat{\Bo}_{\tilde{m}}\hat{f}(\zeta)}{\zeta}e_m\left(\frac{s}{\zeta}\right)\,d\zeta\right]_d,
\end{equation*}
where the contour $\tilde{\gamma}_{s,d}$ is defined in Proposition \ref{prop:5}.
\end{Prop}
\begin{proof}
First, observe that 
by (\ref{eq:sum}) and (\ref{eq:like}) the jump for $\hat{f}$ across the Stokes line $L_d$ may be written by 
\begin{equation}
\label{eq:jump_hyp}
 J_{L_d}\hat{f}(t)=f^{d^+}(t)-f^{d^-}(t)=\big(\mathcal{L}_{\tilde{m},d^+}-\mathcal{L}_{\tilde{m},d^-}\big)\hat{\Bo}_{\tilde{m}}\hat{f}(t)=\big(\mathcal{L}_{\Gamma_{1/k},d^+}-\mathcal{L}_{\Gamma_{1/k},d^-}\big)\Bo_{m^{-1}}\hat{\Bo}_{\tilde{m}}\hat{f}(t).
\end{equation}
Moreover, by (\ref{eq:deform}) we can treat
$$g_0(s):=\Bo_{m^{-1}}\hat{\Bo}_{\tilde{m}}\hat{f}(s)=\frac{1}{2\pi i}\oint_{\tilde{\gamma}_{s,d}}\frac{\hat{\Bo}_{\tilde{m}}\hat{f}(\zeta)}{\zeta}e_m\left(\frac{s}{\zeta}\right)\,d\zeta\in\Oo^{k}(D\cup (S_{d}\setminus L_{d}))$$
as a defining function of the hyperfunction $G(s):=[g_0(s)]_{d}\in \mathcal{H}^{k}(L_{d})$.
So, combining (\ref{eq:lap}), (\ref{eq:kothe}) and (\ref{eq:jump_hyp}) we conclude that 
\begin{gather*}
 J_{L_d}\widehat{f}(t)=G_0(s)\left[\frac{k(s/t)^ke^{-(s/t)^k}}{s}\right]\quad \textrm{for sufficiently small}\ r>0\ \textrm{and}\
 t\in S_{d}(\frac{\pi}{k}, r).
\end{gather*}
\end{proof}

Now, let $\hat{f}\in\CC[[t]]$ be $\mathbf{k}$-multisummable and $\mathbf{d}$ be as in Definition \ref{df:multisummable} with $L_{d_j}$
being the Stokes line of level $k_j$. It means that $\hat{f}=\hat{f}_1+\cdots+\hat{f}_n$, where
$\hat{f}_j$ is $k_j$-summable for $j=1,\dots,n$. 
Then, analogously as in the summable case, the jump across $L_{d_j}$ of level $k_j$ is given by
$$
J_{L_{d_j},k_j}\hat{f}=
f^{\mathbf{d_j^+}}-f^{\mathbf{d_j^-}}=f_j^{d_j^+}-f_j^{d_j^-}
$$
and we may describe this jump in terms of hyperfunctions. Hence we get directly the following multisummable version of Proposition \ref{prop:stokes}.
\begin{Prop}\label{prop:stokes_multi}
 Let $k_1>\dots>k_n>0$ and $\mathbf{k}=(k_1,\dots,k_n)$. We assume that $\hat{f}\in\CC[[t]]$ is $\mathbf{k}$-multisummable, $j\in\{1,\dots,n\}$, $\tilde{m}_j=\Gamma_{1/k_j}m_j$ is a sequence of generalized moments of order $1/k_j$ and $d_j$ is a singular direction of $\hat{f}$ of level $k_j$. Then
\begin{multline*}
 J_{L_{d_j},k_j}\hat{f}(t)=f^{\mathbf{d_j^+}}(t)-f^{\mathbf{d_j^-}}(t)=f_j^{d_j^+}(t)-f_j^{d_j^-}(t)=\big(\mathcal{L}_{\tilde{m}_j,d_j^+}-\mathcal{L}_{\tilde{m}_j,d_j^-}\big)\hat{\Bo}_{\tilde{m}_j}\hat{f}_j(t)\\
 =G_{0j}(s)\left[\frac{k_j(s/t)^{k_j}e^{-(s/t)^{k_j}}}{s}\right]
\end{multline*}
for sufficiently small $r_j>0$ and $t\in S_{d_j}(\frac{\pi}{k_j}, r_j)$, where
the hyperfunction $G_{0j}(s)$ is defined by
\begin{equation*}
 G_{0j}(s)=\left[\Bo_{m^{-1}_j}(\hat{\Bo}_{\tilde{m}_j}\hat{f}_j(s))\right]_{d_j}=
 \left[\frac{1}{2\pi i}\oint_{\tilde{\gamma}_{s,d_j}}\frac{\hat{\Bo}_{\tilde{m}_j}\hat{f}_j(\zeta)}{\zeta}e_{m_j}\left(\frac{s}{\zeta}\right)\,d\zeta\right]_{d_j}.
\end{equation*}
\end{Prop}

\section{Singular directions for summable solutions of some moment-pdes}\label{sec:12}
In this section we explain how to use a kernel function $e_m(z)$ associated with a given sequence preserving summability $m$ to describe jumps across the singular directions for summable and multisummable solutions of some moment partial differential equations.

First we define actual solutions of (\ref{eq:CP_u}) in the similar way to \cite[Definition 8]{M-P} and \cite[Definition 18]{M-T}. Since $\partial^l_{t}\hat{u}(0,z)=\frac{l!}{\tilde{m}_1(l)}\partial^l_{\tilde{m}_1,t}\hat{u}(0,z)$ for every $\hat{u}(t,z)\in\EE[[t]]$, we have
\begin{Def}
Let $G$ be a sectorial region in the universal covering space $\widetilde{\CC}$. A function $u\in\Oo(G,\EE)$ is called an \emph{actual solution}
of (\ref{eq:CP_u})
if it satisfies 
\begin{equation*}
    \left\{
    \begin{array}{l}
     P(\partial_{\tilde{m}_1,t},\partial_{\tilde{m}_2,z})u=0\\
     \lim\limits_{t\to 0,\ t\in G}\partial^l_{t}u(t,z)=\frac{l!}{\tilde{m}_1(l)}\varphi_l(z)\in\EE,\quad l=0,\dots,p-1.
    \end{array}
    \right.
 \end{equation*}
\end{Def}

Now we are ready to formulate the following general result for summable solutions of (\ref{eq:CP_u}).
\begin{Thm}\label{th:5}
 Let $k>0$ and let $\tilde{m}_1$, $\tilde{m}_2$ be sequences of generalized moments. We assume that the solution $\hat{u}\in\Oo[[t]]$ of the Cauchy problem 
 \begin{equation}
 \label{eq:CP_u2}
    \left\{
    \begin{array}{l}
     P(\partial_{\tilde{m}_1,t},\partial_{\tilde{m}_2,z})u=0\\
     \partial^l_{\tilde{m}_1,t}u(0,z)=\varphi_l(z)\in\Oo(D),\quad l=0,\dots,p-1
    \end{array}
    \right.
 \end{equation}
 is $k$-summable with singular directions $d_1,\dots,d_n$ with $0\leq d_1<\dots<d_n<2\pi$ and $\tilde{m}=m\Gamma_{1/k}$ is a sequence of generalized moments of order $1/k$.
 Then $v=\hat{\Bo}_{\tilde{m},t}\hat{u}\in\Oo(D\times D)$ is a holomorphic solution of the Cauchy problem
 \begin{equation}
\label{eq:CP_v}
    \left\{
    \begin{array}{l}
     P(\partial_{\tilde{m}\tilde{m}_1,t},\partial_{\tilde{m}_2,z})v=0\\
     \partial^l_{\tilde{m}\tilde{m}_1,t}v(0,z)=\varphi_l(z)\in\Oo(D),\quad l=0,\dots,p-1.
    \end{array}
    \right.
 \end{equation}
 
 Moreover, using this holomorphic solution $v\in\Oo(D\times D)$ of (\ref{eq:CP_v}) we can get a family of $n$ actual solutions
 $\{u_1,\dots,u_n\}$ of (\ref{eq:CP_u2}), where
 \begin{equation}\label{eq:u_j}
  u_j(t,z)=u^{\theta}(t,z)=\La_{\tilde{m},\theta}v(t,z)\quad\text{for}\quad \theta\in (d_{j-1},d_j),\quad j=1,\dots,n,
 \end{equation}
 $d_j$ are singular directions and $d_0=d_n-2\pi$,
 and $u_j\in\Oo(G_j\times D)$, where $G_j=G_j(\frac{d_{j-1}+d_j}{2}, \pi/k+d_j-d_{j-1})$  is a 
 sectorial region of opening $\pi/k+d_j-d_{j-1}>0$ and bisecting direction $\frac{d_{j-1}+d_j}{2}\in\RR$.
 
 Furthermore, a jump for $\hat{u}$ across the singular direction $d_j$ ($j=1,\dots,n$) is given by
 \begin{equation*}
  J_{L_{d_j},k}\hat{u}(t,0)=u^{d_j^+}(t,0)-u^{d_j^-}(t,0)=\left[\frac{1}{2\pi i}\oint_{\tilde{\gamma}_{s,d_j}}\frac{v(\zeta,0)}{\zeta}e_m\left(\frac{s}{\zeta}\right)\,d\zeta\right]_{d_j}\left[\frac{k(s/t)^ke^{-(s/t)^k}}{s}\right].
 \end{equation*}
\end{Thm}
\begin{proof}
 Since the solution $\hat{u}$ of (\ref{eq:CP_u2}) is $k$-summable and $\tilde{m}$ is a sequence of generalized moments of order $1/k$ then by Remark \ref{re:summable} we see that $v=\hat{\Bo}_{\tilde{m},t}\hat{u}\in\Oo(D\times D)$. Moreover, by Theorem~\ref{th:4} it follows that $v$ is a solution of the Cauchy problem (\ref{eq:CP_v}).
 
 Next, by Remark \ref{re:laplace} we deduce that $u^{\theta_1}\in\Oo(G(\theta_1, \pi/k)\times D)$ and $u^{\theta_2}\in\Oo(G(\theta_2, \pi/k)\times D)$ for every $\theta_1,\theta_2\in (d_{j-1},d_j)$. Moreover, by \cite[Lemma 10]{B2} $u^{\theta_1}=u^{\theta_2}$ on $\big(G(\theta_1, \pi/k)\times D\big)\cap\big( G(\theta_2, \pi/k)\times D\big)$. Hence the function $u_j$ given by (\ref{eq:u_j}) is well-defined and is analytically continued to the set $G_j=G_j(\frac{d_{j-1}+d_j}{2}, \pi/k+d_j-d_{j-1})$ with respect to $t$.
 Moreover, by Remark \ref{re:sum} we see that $u_j(t,z)=u^{\theta}(t,z)=\Ss_{k,\theta}\hat{u}(t,z)$.
 
 By the moment version of \cite[Theorem 3.2]{B1} we conclude that the space of $k$-summable series in a direction $d$ is a moment differential algebra over $\CC$.
 It means that it is a linear space, which is also closed under multiplication
 and moment differentiations, and which for any series $\hat{f},\hat{g}\in\EE\{t\}_{k,d}$ satisfies:
 $\Ss_{k,d}(\hat{f}+\hat{g})=\Ss_{k,d}\hat{f}+\Ss_{k,d}\hat{g}$, $\Ss_{k,d}(\hat{f}\cdot\hat{g})=\Ss_{k,d}\hat{f}\cdot\Ss_{k,d}\hat{g}$,
 $\Ss_{k,d}(\partial_{\tilde{m}_1,t}\hat{f})=
 \partial_{\tilde{m}_1,t}(\Ss_{k,d}\hat{f})$
 and $\Ss_{k,d}(\partial_{\tilde{m}_2,z}\hat{f})=
 \partial_{\tilde{m}_2,z}(\Ss_{k,d}\hat{f}$).
 
 Hence for every $\theta\in (d_{j-1},d_j)$
  \begin{gather*}
  P(\partial_{\tilde{m}_1,t},\partial_{\tilde{m}_2,z})u_{j}=P(\partial_{\tilde{m}_1,t},\partial_{\tilde{m}_2,z})\Ss_{k,\theta}\hat{u}=
  \Ss_{k,\theta}P(\partial_{\tilde{m}_1,t},\partial_{\tilde{m}_2,z})\hat{u}=0.
 \end{gather*}
 Additionally, if we write the formal solution $\hat{u}\in\Oo\{t\}_k$ of (\ref{eq:CP_u2}) as $\hat{u}(t,z)=\sum_{n=0}^{\infty}u_n(z)t^n$ then by \cite[Proposition 8]{B2} we get
 \begin{gather*}
   \lim\limits_{t\to 0, t\in G_j}\partial_{t}^lu_{j}(t,z)=l!u_l(z)=\frac{l!}{\tilde{m}_1(l)}\partial^l_{\tilde{m}_1,t}\hat{u}(0,z)=\frac{l!}{\tilde{m}_1(l)}\varphi_l(z)\quad \textrm{for}\quad l=0,\dots,p-1.
 \end{gather*}
 Therefore $u_{j}(t,z)$ is an actual solution of (\ref{eq:CP_u2}) for $j=1,\dots,n$.
 
 By Proposition \ref{prop:stokes} we get 
 \begin{equation*}
  J_{L_{d_j},k}\hat{u}(t,0)=u^{d_j^+}(t,0)-u^{d_j^-}(t,0)=\left[\frac{1}{2\pi i}\oint_{\tilde{\gamma}_{s,d_j}}\frac{\hat{\Bo}_{\tilde{m}}u(\zeta,0)}{\zeta}e_m\left(\frac{s}{\zeta}\right)\,d\zeta\right]_{d_j}\left[\frac{k(s/t)^ke^{-(s/t)^k}}{s}\right].
 \end{equation*}
 The proof is completed by observation that $\hat{\Bo}_{\tilde{m}}u(\zeta,0)=v(\zeta,0)$.
\end{proof}

To be able to calculate jumps more precisely, we have to limit ourselves to simpler equations.
\begin{Thm}\label{th:6}
 Let $\tilde{m}_1=\Gamma_{s_1}m_1$ and $\tilde{m}_2=\Gamma_{s_2}m_2$ be sequences of generalized moments of orders $s_1$ and $s_2$ respectively, where $0\leq s_1<s_2$. We also assume that $a,\lambda\in\CC$, $z_0\in\CC\setminus\{0\}$ and $k=\frac{1}{s_2-s_1}$.
 Then the Cauchy problem
 \begin{equation}
 \label{eq:first}
    \left\{
    \begin{array}{l}
     \big(\partial_{\tilde{m}_1,t}-\lambda\partial_{\tilde{m}_2,z}\big)u=0\\
     u(0,z)=\varphi(z):=\frac{a}{z-z_0}\in\Oo(D)
    \end{array}
    \right.
 \end{equation}
 has a unique formal power series solution $\hat{u}\in\Oo[[t]]$ and this solution is $k$-summable with only one singular direction $d=\arg z_0 - \arg\lambda$ (modulo $2\pi$). 
 Moreover, a jump for $\hat{u}$ across this singular direction $d$ is given by
 \begin{equation}
  J_{L_d,k}\hat{u}(t,0)=u^{d^+}(t,0)-u^{d^-}(t,0)=\left[\frac{a}{z_0}e_{m}\Big(\frac{\lambda s}{z_0}\Big)\right]_d\left[\frac{k(s/t)^ke^{-(s/t)^k}}{s}\right],
 \end{equation}
where $e_m(z)$ denotes a kernel function associated with a sequence preserving summability
\begin{equation}\label{eq:m}
m=\frac{\Gamma_{s_2}}{\Gamma_{s_1}\Gamma_{s_2-s_1}}\frac{m_2}{m_1}.
\end{equation}
\end{Thm}
\begin{proof}
First observe that 
\begin{equation*}
\hat{u}(t,z)=\sum_{n=0}^{\infty}\frac{\lambda^n\partial^n_{\tilde{m_2},z}\varphi(z)}{\tilde{m}_1(n)}
\end{equation*}
is the unique formal power series solution of (\ref{eq:first})
Since 
$$\hat{\Bo}_{m_2^{-1},z}\frac{a}{z-z_0}=-\frac{a}{z_0}\hat{\Bo}_{m_2^{-1},z}\left(\sum_{n=0}^{\infty}\Big(\frac{z}{z_0}\Big)\right)=-\frac{a}{z_0}e_{m_2}(z/z_0),
$$
applying Theorem \ref{th:4} to (\ref{eq:first}) we conclude that $\hat{v}=\hat{\Bo}_{m_1^{-1},t}\hat{\Bo}_{m_2^{-1},z}\hat{u}$ 
is the formal solution of the Cauchy problem
\begin{equation*}
    \left\{
    \begin{array}{l}
     \big(\partial_{\Gamma_{s_1},t}-\lambda\partial_{\Gamma_{s_2},z}\big)v=0\\
     v(0,z)=\hat{\Bo}_{m_2^{-1},z}\varphi(z)=-\frac{a}{z_0}e_{m_2}(z/z_0).
    \end{array}
    \right.
 \end{equation*}
Next, let $\hat{w}=\hat{\Bo}_{\Gamma_{s_2}/\Gamma_{s_1},t}\hat{v}$. Then $\hat{w}$ is the formal solution of the Cauchy problem
\begin{equation*}
    \left\{
    \begin{array}{l}
     \big(\partial_{\Gamma_{s_2},t}-\lambda\partial_{\Gamma_{s_2},z}\big)w=0\\
     w(0,z)=\hat{\Bo}_{m_2^{-1},z}\varphi(z)=-\frac{a}{z_0}e_{m_2}(z/z_0).
    \end{array}
    \right.
\end{equation*}

 Let $k=\frac{1}{s_2-s_1}$. Since the function $z\mapsto -\frac{a}{z_0}e_{m_2}(z/z_0)$ belongs to the space $\Oo^k(\hat{S}_{\theta})$ for every direction $\theta\neq \arg z_0$ (modulo $2\pi$), by \cite[Lemma 5]{Mic7} we conclude that $w\in\Oo^k(\hat{S}_{\theta-\arg\lambda},\Oo)=:\Oo^k(\hat{S}_{\theta-\arg\lambda}\times D)$ for such $\theta$. Hence $\hat{v}$ is $k$-summable with only one singular direction $d=\arg z_0-\arg\lambda$. By Theorem \ref{th:4} it means that also $\hat{u}$ is $k$-summable with only one singular direction $d=\arg z_0-\arg\lambda$.
 
 Applying Theorem \ref{th:5} to $k$-summable formal power series solution $\hat{u}$ of the Cauchy problem (\ref{eq:first}) with the singular direction $d=\arg z_0-\arg\lambda$ and with $m$ given by (\ref{eq:m}), we conclude that
 \begin{equation*}
  J_{L_{d},k}\hat{u}(t,0)=u^{d^+}(t,0)-u^{d^-}(t,0)=\left[\frac{1}{2\pi i}\oint_{\tilde{\gamma}_{s,d}}\frac{V(\zeta,0)}{\zeta}e_m\left(\frac{s}{\zeta}\right)\,d\zeta\right]_{d}\left[\frac{k(s/t)^ke^{-(s/t)^k}}{s}\right],
 \end{equation*}
 where $V=\hat{\Bo}_{m\Gamma_{s_2-s_1},t}\hat{u}\in\Oo(D\times D)$ is a holomorphic solution of the Cauchy problem
 \begin{equation*}
    \left\{
    \begin{array}{l}
     \big(\partial_{\tilde{m}_2,t}-\lambda\partial_{\tilde{m}_2,z}\big)V=0\\
     V(0,z)=\varphi(z):=\frac{a}{z-z_0}
    \end{array}
    \right..
 \end{equation*}
 
 Since 
 \begin{equation*}
  V(t,z)=\sum_{n=0}^{\infty}\frac{\lambda^n\partial_{\tilde{m}_2,z}^n\varphi(z)}{\tilde{m}_2(n)}t^n
 \end{equation*}
we get
\begin{equation*}
  V(t,0)=\sum_{n=0}^{\infty}\frac{\partial_{\tilde{m}_2,z}^n\varphi(0)}{\tilde{m}_2(n)}(\lambda t)^n=\varphi(\lambda t)=\frac{a}{\lambda t-z_0}.
 \end{equation*}
 Hence by the residue theorem
 \begin{equation*}
  \frac{1}{2\pi i}\oint_{\tilde{\gamma}_{s,d}}\frac{V(\zeta,0)}{\zeta}e_m\left(\frac{s}{\zeta}\right)\,d\zeta=\frac{1}{2\pi i}\oint_{\tilde{\gamma}_{s,d}}\frac{a}{\lambda\zeta(\zeta-z_0/\lambda)}e_m\left(\frac{s}{\zeta}\right)\,d\zeta=\frac{a}{z_0}e_m\left(\frac{\lambda s}{z_0}\right),
 \end{equation*}
which completes the proof.
\end{proof}

\begin{Ex}\label{ex:4}
Let us consider the Cauchy problem
\begin{equation}
\label{eq:special}
    \left\{
    \begin{array}{l}
     \big(D_{q,t}-\lambda\partial_{z}\big)u=0\\
     u(0,z)=\varphi(z):=\frac{a}{z-z_0},
    \end{array}
    \right.
 \end{equation}
 where $D_{q,t}$ denotes the $q$-difference operator and $q\in(0,1)$. Observe that (\ref{eq:special}) is a special case of (\ref{eq:first}) with $s_1=0$, $m_1=([n]_q!)_{n\geq 0}$, $s_2=1$ and $m_2=\mathbf{1}=(1)_{n\geq 0}$. It means that $k=1$ and $m=m_1^{-1}$. By Theorem \ref{th:6} the unique formal solution of (\ref{eq:special}) is $1$-summable with a unique singular direction $d=\arg z_0-\arg\lambda$. Moreover, since 
 $$e_m(z)=\sum_{n=0}^{\infty}\frac{z^n}{[n]_q!}=\exp_q(z)=\sum_{n=0}^{\infty}\frac{(1-q)^n}{(q;q)_n}z^n=\prod_{n=0}^{\infty}\frac{1}{1-(1-q)q^nz}.$$
 is meromorphic on $\CC$ (see Example \ref{ex:2}) with simple poles at
 \begin{equation*}
  z=z_n=\frac{q^{-n}}{1-q}\quad\text{for}\quad n\in\NN_0,
 \end{equation*}
by the residue theorem we conclude that (cf. \cite[the proof of Lemma 2]{I-M})
\begin{multline*}
  J_{L_d,1}\hat{u}(t,0)=u^{d^+}(t,0)-u^{d^-}(t,0)=\left[\frac{a}{z_0}e_{m}\Big(\frac{\lambda s}{z_0}\Big)\right]_d\left[\frac{(s/t)e^{-s/t}}{s}\right]
  =\left[\frac{a}{z_0}\exp_q\Big(\frac{\lambda s}{z_0}\Big)\right]_d\left[\frac{e^{-s/t}}{t}\right]\\
  =-\frac{a}{z_0t}\sum_{n=0}^{\infty}\exp\left(-\frac{z_0}{\lambda t}\frac{q^{-n}}{1-q}\right)\mathrm{res}_{z=z_n}\exp_q(z)=-\frac{a}{z_0t}\sum_{n=0}^{\infty}\exp\left(-\frac{z_0}{\lambda t}\frac{q^{-n}}{1-q}\right)\frac{(-1)^nq^{\frac{n(n+1)}{2}}}{(q;q)_n(q;q)_{\infty}}.
 \end{multline*}
\end{Ex}

More general we have
\begin{Ex}\label{ex:5}
Fix $\tilde{s}>0$. We consider the following generalization of (\ref{eq:special})
\begin{equation}
\label{eq:special2}
    \left\{
    \begin{array}{l}
     \big(D_{q,t}-\lambda\partial_{\Gamma_{\tilde{s}},z}\big)u=0\\
     u(0,z)=\varphi(z):=\frac{a}{z-z_0},
    \end{array}
    \right..
 \end{equation}
 We see that (\ref{eq:special2}) is also a special case of (\ref{eq:first}) with $s_1=0$, $m_1=([n]_q!)_{n\geq 0}$, $s_2=\tilde{s}$ and $m_2=\mathbf{1}=(1)_{n\geq 0}$. It means that $k=1/\tilde{s}$ and $m=m_1^{-1}$. By Theorem \ref{th:6} the unique formal solution of (\ref{eq:special2}) is $k$-summable with a unique singular direction $d=\arg z_0-\arg\lambda$. As previously
 $e_m(z)=\exp_q(z)$
 is meromorphic on $\CC$ with simple poles at
 $z=z_n=\frac{q^{-n}}{1-q}$ for $n\in\NN_0$.
By the residue theorem we get
\begin{multline*}
  J_{L_d,k}\hat{u}(t,0)=u^{d^+}(t,0)-u^{d^-}(t,0)=\left[\frac{a}{z_0}e_{m}\Big(\frac{\lambda s}{z_0}\Big)\right]_d\left[\frac{k(s/t)^ke^{-(s/t)^k}}{s}\right]\\
  =\left[\frac{a}{z_0}\exp_q\left(\frac{\lambda s}{z_0}\right)\right]_d\left[\frac{k(s/t)^ke^{-(s/t)^k}}{s}\right]
  =-\frac{ak}{z_0t^k}\sum_{n=0}^{\infty}\left(\frac{z_0}{\lambda}\frac{q^{-n}}{1-q}\right)^{k-1}e^{-\left(\frac{z_0}{\lambda t}\frac{q^{-n}}{1-q}\right)^k}\mathrm{res}_{z=z_n}\exp_q(z)\\
  =-\frac{akz_0^{k-2}}{\lambda^{k-1}(1-q)^{k-1}t^k}\sum_{n=0}^{\infty}q^{-n(k-1)}e^{\left(-\frac{z_0}{\lambda t}\frac{q^{-n}}{1-q}\right)^k}\frac{(-1)^nq^{\frac{n(n+1)}{2}}}{(q;q)_n(q;q)_{\infty}}.
 \end{multline*}
\end{Ex}

In the last example we consider multisummable solution of the following equation
\begin{Ex}\label{ex:6}
As previously we assume that
$$\varphi(z)=\frac{a}{z-z_0}=-\sum_{n=0}^{\infty}\frac{a}{z_0^{n+1}}z^n.$$
We consider a formal solution $\hat{u}$ of the Cauchy problem
\begin{equation}\displaystyle
\label{eq:special3}
    \left\{
    \begin{array}{l}
     \big(D_{q,t}-\lambda_1\partial_{\Gamma_{1/2},z}\big)\big(D_{q,t}-\lambda_2\partial_{z}\big)u=0\\
     u(0,z)=2\varphi(z)=\frac{2a}{z-z_0},\\
     D_{q,t}u(0,z)=\lambda_1\partial_{\Gamma_{1/2},z}\varphi(z)+\lambda_2\partial_z\varphi(z)=-\sum_{n=0}^{\infty}\frac{\Gamma(1+(n+1)/2)a\lambda_1}{\Gamma(1+n/2)z_0^{n+2}}z^n-\frac{a\lambda_2}{(z-z_0)^2}.
    \end{array}
    \right.
 \end{equation}
 Observe that $\hat{u}=\hat{u}_1+\hat{u}_2$, where $\hat{u}_1$ is a formal solution of
 \begin{equation}\displaystyle
\label{eq:special4}
   \left\{
    \begin{array}{l}
     \big(D_{q,t}-\lambda_1\partial_{\Gamma_{1/2},z}\big)u_1=0\\
     u_1(0,z)=\varphi(z)=\frac{a}{z-z_0}
    \end{array}
    \right. 
 \end{equation}
 and $\hat{u}_2$ is a formal solution of 
 \begin{equation}\displaystyle
\label{eq:special5}
    \left\{
    \begin{array}{l}
     \big(D_{q,t}-\lambda_2\partial_{z}\big)u_2=0\\
     u_2(0,z)=\varphi(z)=\frac{a}{z-z_0}
    \end{array}
    \right..
 \end{equation}
By the previous examples $\hat{u}_1$ is $2$-summable with a unique singular direction $d_1=\arg z_0 - \arg\lambda_1$ and $\hat{u}_2$ is $1$-summable with a unique singular direction $d_2=\arg z_0 - \arg\lambda_2$. Therefore $\hat{u}$ is $(2,1)$-multisummable, $d_1$ is a singular direction of $\hat{u}$ of level $2$ and $d_2$ is a singular direction of $\hat{u}$ of level $1$. Moreover
\begin{equation*}
  J_{L_{d_1},2}\hat{u}(t,0)
  =-\frac{2a}{\lambda_1(1-q)t^2}\sum_{n=0}^{\infty}q^{-n}e^{\left(-\frac{z_0}{\lambda_1 t}\frac{q^{-n}}{1-q}\right)^2}\frac{(-1)^nq^{\frac{n(n+1)}{2}}}{(q;q)_n(q;q)_{\infty}}
\end{equation*}
and
\begin{equation*}
 J_{L_{d_2},1}\hat{u}(t,0)=-\frac{a}{z_0t}\sum_{n=0}^{\infty}\exp\left(-\frac{z_0}{\lambda_2 t}\frac{q^{-n}}{1-q}\right)\frac{(-1)^nq^{\frac{n(n+1)}{2}}}{(q;q)_n(q;q)_{\infty}}.
\end{equation*}
\end{Ex}

\end{document}